\documentclass{amsart}
\usepackage{amsmath,amsthm}
\usepackage{amsfonts,amssymb}
\usepackage{accents}
\usepackage{enumerate}
\usepackage{accents,color}
\usepackage{graphicx}
\usepackage{wrapfig}
\newcommand{\lvt}{\left|\kern-1.35pt\left|\kern-1.3pt\left|}
\newcommand{\rvt}{\right|\kern-1.3pt\right|\kern-1.35pt\right|}

\newtheorem{thm}{Theorem}[section]
\newtheorem{cor}[thm]{Corollary}
\newtheorem{lem}[thm]{Lemma}
\newtheorem{prop}[thm]{Proposition}

\newtheorem{defn}[thm]{Definition}

\theoremstyle{remark}
\newtheorem{rem}{Remark}[section]

 \def\la{{\langle}}
 \def\ra{{\rangle}}
 
\def\one{{\mathbf 1}}
 \def\sph{{\mathbb{S}^{d-1}}}
 
 \def\sb{{\mathsf b}}
 
 \def\db{{\mathsf d}}
 \def\sh{{\mathsf h}}
 
 \def\db{{\mathsf D}}

 \def\sK{{\mathsf K}}
 \def\sT{{\mathsf T}}
 \def\sP{{\mathsf P}}
 \def\sQ{{\mathsf Q}}
 \def\sS{{\mathsf S}}

 \def\d{{\mathrm{d}}}
  
 \def\a{{\alpha}}
 \def\b{{\beta}}
 \def\g{{\gamma}}

 \def\l{{\lambda}}
 \def\o{{\omega}}
 \def\s{\sigma}
 \def\la{{\langle}}
 \def\ra{{\rangle}}

 \def\bb{{\mathbf b}}
 
 \def\db{{\mathbf d}}
 \def\kb{{\mathbf k}}
 \def\hb{{\mathbf h}}

 \def\Kb{{\mathbf K}}
 \def\Pb{{\mathbf P}}
 \def\Qb{{\mathbf Q}}
 \def\Sb{{\mathbf S}}
 \def\Tb{{\mathbf T}}

 \def\CH{{\mathcal H}}

 \def\CT{{\mathcal T}}
 
 \def\CV{{\mathcal V}}

 \def\BB{{\mathbb B}}

 \def\NN{{\mathbb N}}

 \def\RR{{\mathbb R}}
 \def\SS{{\mathbb S}}
 
 \def\UU{{\mathbb U}}
 \def\VV{{\mathbb V}}
 
      \def\proj{\operatorname{proj}}

\def\lla{\langle{\kern-2.5pt}\langle}      
\def\rra{\rangle{\kern-2.5pt}\rangle}

\def\f{\frac}

\graphicspath{{./}}
\begin{document}

\title{Fourier orthogonal series on a paraboloid}
 
\author{Yuan Xu}
\address{Department of Mathematics\\ University of Oregon\\
    Eugene, Oregon 97403-1222.}\email{yuan@uoregon.edu}

\date{\today}
\keywords{Orthogonal polynomials, Fourier orthogonal series, surface, paraboloid}
\subjclass[2010]{42C05, 42C10, 33C50}

\begin{abstract} 
We study orthogonal structures and Fourier orthogonal series on the surface of a paraboloid 
$\VV_0^{d+1} = \{(x,t): \|x\| = \sqrt{t}, \, x \in \RR^d, \, 0\le t<1\}$. The reproducing kernels of the 
orthogonal polynomials with respect to $t^\b(1-t)^\g$ on $\VV_0^{d+1}$ are related to the 
reproducing kernels of the Jacobi 
polynomials on the parabolic domain $\{(x_1,x_2): x_1^2 \le x_2 \le 1\}$ in $\RR^2$. 
This connection serves as an essential tool for our study of the Fourier orthogonal series 
on the surface of the paraboloid, which allow us, in particular, to study the convergence
of the Ces\`aro means on the surface. Analogous results are also established for the 
solid paraboloid bounded by $\VV_0^{d+1}$ and the hyperplane $t=1$. 
\end{abstract}

\maketitle

\section{Introduction}
\setcounter{equation}{0}
 
The Laplace series, so named the generalized Fourier series in spherical harmonics on the unit sphere, has 
been extensively studied. One essential ingredient for understanding these series is the addition formula for 
spherical harmonics, which states that the reproducing kernel of the orthogonal projection operator from $L^2(\sph)$ 
onto the space of spherical harmonics of degree $n$ can be written as $Z_n (\la \xi,\eta\ra)$, where $Z_n$ 
is a Gegenbauer polynomial of degree $n$ in one variable and $\la \xi,\eta\ra$ is the Euclidean inner product 
of $\xi,\eta \in \sph$. Because of this closed formula of the reproducing kernels, much of the study of the 
Laplace series can be reduced to the study of the Fourier-Gegenbauer series of one variable \cite{AF, DaiX, SW}. 

The above narrative turns out to be the prototype for Fourier orthogonal series domains in 
higher dimensions. In the past two decades, starting from the addition formula for classical orthogonal 
polynomials on the unit ball \cite{X99}, closed form formulas for reproducing kernels of orthogonal 
polynomials have been discovered for several regular domains, including the unit ball, regular simplex, 
cylinder, as well as the unit sphere with inner product defined by weighted integrals, which makes study 
of the Fourier orthogonal series on these domain feasible; see, for example, \cite{OC, CFX, DaiX, DX, 
KePX, KL, KPX, SS, Wade, W, WZ, X99, X15} and their references. For unbounded classical domains, 
we refer to \cite{Tha} as well as to \cite{BRT, Tha2} for references on more recent works, which however 
require techniques beyond our narrative. 

To step beyond the regular domains, we recently started to analyze orthogonal structure on quadratic 
surfaces of revolutions other than the unit sphere as well as on domains bounded by such quadratic
surfaces. Let $\VV_0^{d+1}$ be a quadratic surface in $\RR^{d+1}$, parametrized in $(x,t)$, $x \in \RR^d$ 
and $t\in \RR$, that is a surface of revolution around the $t$ axis. We consider orthogonal structure defined by 
the inner product 
$$
   \la f,g\ra = \int_{\VV_0^{d+1}} f(x,t) g(x,t) \varpi (t) \d\s(x,t),
$$
where $\varpi(t)$ is a weight function and $\d\s$ is the Lebesgue measure of $\VV_0^{d+1}$. Taking the cue 
of spherical harmonics, we look for families of orthogonal polynomials that share two characteristic properties 
of spherical harmonics, one is an addition formula and the other is the existence of a second order partial 
differential operator that has orthogonal polynomials as eigenfunctions with eigenvalues depending only 
on the degree of the polynomials, which is an analogue of the Laplace-Beltrami operator on the unit sphere.

In \cite{X19} we studied orthogonal polynomials on the surface of the cone
$$
       \VV_0^{d+1} = \left \{(x,t): \|x\| = t, \, 0 \le t \le b, \, x\in \RR^d \right\},
$$ 
where $b=1$ or $b = \infty$, and identified two families of orthogonal polynomials that are eigenfunctions of 
a differential operator, the Laguerre polynomials on the cone with $b = +\infty$ and the Jacobi polynomials on 
the cone with $b=1$, and also established that the Jacobi family possesses an addition formula, which allows  
us to carry out a preliminary study of the Fourier orthogonal series on the surface of the cone. Moreover, 
analogous results were also established on the domain bounded by the surface of the cone, together with the 
hyperplane $t=1$ when $b=1$. In \cite{X20}, we considered the orthogonal structure on the surface of the 
hyperboloid
$$
       \VV_0^{d+1} = \left\{(x,t): \|x\|^2 = c^2 (t^2 - \varrho^2),\,  x\in \RR^d, \, \varrho \le |t| \le b\right\},
$$ 
where $\varrho \ge 0$ and  $b=1+\varrho$ or $\infty$, and it degenerates to the double cone when $\varrho =0$. 
In this case the weight function $\varpi$ is an even function. We again identified two families 
of orthogonal polynomials, the Hermite polynomials on the hyperboloid with $b=\infty$ and the Gegenbauer 
polynomials on the hyperboloid with $b= 1 + \rho$. However, for these two families, only those polynomials 
that are even in $t$ are eigenfunctions of a differential operator. Furthermore, the addition formula holds 
for the Gegenbauer polynomials on the hyperboloid that are even in $t$. These results allowes us to carry out 
a study of the Fourier orthogonal series for functions that are even in $t$ over the hyperboloid. Moreover, 
analogous results were also established on the domain bounded by the surface of the hyperboloid, together 
with the hyperplane $t=1$ when $b=1+\varrho$. 

In the present paper, we study orthogonal structure on a paraboloid of revolution, which turns out to be 
very different from those on the cone and on the hyperboloid. We shall consider the surface of paraboloid 
defined by 
$$
       \VV_0^{d+1} = \left\{(x,t): \|x\| = \sqrt{t},\,  x\in \RR^d, \,  0 \le t \le b\right\}
$$
as well as the solid paraboloid $\VV^{d+1}$ bounded by $\VV_0^{d+1}$ and the hyperplane $t=b$. For
our study we shall consider only $b =1$, or the compact case, since there is not as much that makes 
the case $b=\infty$ standing out for paraboloids.  
 
On the surface of the paraboloid, we consider a family of orthogonal polynomials with respect to the 
weight function $\varpi(t) = t^\b(1-t)^\g$, which shall be called the Jacobi polynomials on the paraboloid. 
These polynomials will be shown to be eigenfunctions of a second order differential operator but with the 
eigenvalues depending on two indices, the degree of the polynomials as well as another index that 
depends on the particular choice of the orthogonal basis, in contrast to the cone and the hyperboloid for 
which the corresponding eigenvalue property is independent of the choice of orthogonal bases. The 
Jacobi polynomials on the paraboloid also do not possess an explicit addition formula. What they do 
have is a connection to an orthogonal structure on the parabolic domain 
$\UU = \{(x_1,x_2): x_1^2 \le x_2 \le 1\}$ in $\RR^2$, bounded by the parabola $x_2 =x_1^2$ 
and the line $x_2 =1$. A family of orthogonal polynomials on $\UU$, called the Jacobi polynomials on
the parabolic domain, was first considered in \cite{K75} and they satisfy a product formula as shown 
in \cite{Ko97}. The latter was used to study the Fourier orthogonal series in \cite{CFX}, where the 
essential ingredient is an addition formula that holds when one argument of the reproducing kernel is 
at the corner $(1,1)$ of the domain. Our essential realization is that the reproducing kernel of the 
orthogonal polynomials on the paraboloid can be expressed in terms of the reproducing kernel on the 
parabola domain, which provides us with the tool for studying the Fourier orthogonal series on the 
paraboloid. In particular, it allows us to study the convergence of the Ces\`aro means of the series. 
We shall also show that the connection to the structure on $\UU$ also extends to the solid paraboloid, 
which allows us to carry out our study of the Fourier orthogonal series on the solid paraboloid. 

The paper is organized as follows. In the next section, we review orthogonal structure on the parabolic
domain, where enough detail will be provided to prepare for their usage in latter sections. The orthogonal
structure and the Fourier series on the surface of the paraboloid will be discussed in Section 3, and 
analogous results on the solid paraboloid will be discussed in Section 4.  
 
 
\section{Orthogonal polynomials on a parabolic domain}
\setcounter{equation}{0}

As mentioned in the introduction, our development on $\VV^{d+1}$ depends heavily on what is 
known on the parabolic domain 
$$
\UU = \{x \in \RR^2: x_1^2 \le x_2 \le 1\}, 
$$ 
bounded by the parabola $x_2 = x_1^2$ and the line $x_2=1$, which we review in this section. 
For $a > -1$ and $b > -\f12$, we define the weight function
$$
    U_{a,b}(x) = (1-x_2)^{a} (x_2-x_1^2)^{b-\f12}
$$ 
on $\UU$ and consider orthogonal polynomials with respect to the inner product 
\begin{equation}\label{eq:intBB}
\la f,g\ra_\UU:= \db_{a,b} \int_\UU f(x ) g(x ) U_{a,b}(x ) \d x,
\end{equation}
where the normalization constant $\db_{a,b}$
is chosen so that $\la 1,1\ra_{a,b} = 1$ and its value can be verified by writing the integral over $\UU$ as 
\begin{equation} \label{eq:integralBB}
  \int_\UU f(x_1,x_2) \d x_1 \d x_2 = \int_0^1\int_{-1}^1 f\big (u \sqrt{x_2},x_2 \big) \sqrt{x_2} \d u \d x_2. 
\end{equation}
An orthogonal basis for this inner product was defined in \cite{K75} in terms of the Jacobi polynomials. 
For $\a, \b > -1$, the Jacobi polynomials $P_n^{(\a,\b)}$ are defined by 
$$
  P_n^{(\a,\b)}(t) = \frac{(\a+1)_n}{n!} {}_2F_1 \left (\begin{matrix} -n, n+\a+\b+1 \\
      \a+1 \end{matrix}; \frac{1-t}{2} \right),
$$
and they are orthogonal with respect to the weight function $w_{\a,\b}(x) := (1-x)^\a(1+x)^\b$,  
$$
  c_{\a,\b}' \int_{-1}^1 P_n^{(\a,\b)}(t) P_m^{(\a,\b)}(t) w_{\a,\b}(t) \d t = h_n^{(\a,\b)} \delta_{n,m},
$$
where  $c'_{\a,\b} = 2^{-\a-\b-1} c_{\a,\b}$ with 
\begin{align} \label{eq:JacobiNorm}
  c_{\a,\b} := \frac{\Gamma(\a+\b+2)}{\Gamma(\a+1)\Gamma(\b+1)}\quad\hbox{and}\quad  h_n^{(\a,\b)}:= \frac{(\a+1)_n (\b+1)_n(\a+\b+n+1)}{n!(\a+\b+2)_n(\a+\b+2 n+1)}.
\end{align}
In terms of $c_{\a,\b}$, the constant $\db_{a,b}$ in \eqref{eq:intBB} is given by $\db_{a,b} = c_{b-\f12,b-\f12} c_{b,a}$.
 
Let $\CV_n(\UU,U_{a,b})$, $n=0,1,\ldots$, be the space of orthogonal polynomials of degree $n$
with respect to the inner product $\la \cdot,\cdot\ra_{\UU}$ on the parabolic domain. Then
$
  \dim \CV_n(\UU,U_{a,b}) = n+1. 
$
The orthogonal basis for $\CV_n(\UU,U_{a,b})$ given in \cite{K75} consists of polynomials
\begin{equation}\label{eq:U-basis}
  P_{k,n}^{a,b}(x_1,x_2) = P_{n-k}^{(b+k,a)}(1-2x_2) x_2^{\frac{k}2} P_k^{(b-\f12,b-\f12)}\left(\frac{x_1}{\sqrt{x_2}}\right), \quad 0 \le k \le n. 
\end{equation}
Their orthogonality can be verified by using \eqref{eq:integralBB} and so are their $L^2$ norm. In particular, in
terms of the quantities in \eqref{eq:JacobiNorm}, 
\begin{equation}\label {eq:U-basisNorm}
   h_{k,n}^{a,b} = \db_{a,b} \int_\Omega |P_{k,n}^{a,b}(x)|^2 U_{a,b}(x) \d x = 
       \frac{c_{b,a}}{c_{b+k,a}} h_{n-k}^{(b+k,a)} h_k^{(b-\f12,b-\f12)}.
\end{equation}
The polynomials $P_{k,n}^{(\a,\b)}$ satisfy a product formula due to Koornwinder and Schwartz \cite{Ko97}.
The formula is rather complicated and takes the following form: 

\begin{thm} \label{thm:prod-formulaUU}
Let $a \geq b \geq 0$.  For $x=(x_1,x_2), y=(y_1,y_2)\in \UU$, 
\begin{align}\label{eq:prod-formulaUU}
& P_{n,k}^{a,b}(x_1,x_2)\, P_{n,k}^{a,b}(y_1,y_2) \\
    & \quad\qquad  = P_{n,k}^{a,b}(1,1)\, 
    \int_{[0,1]\times [0,\pi]^3}P_{n,k}^{a,b} \big(\xi_1(x,y; r, \psi_1), \xi_2(x,y; r,\psi)\big)\d m_{a,b}(r,\psi), \notag
\end{align}
where $\psi = (\psi_1,\psi_2,\psi_3)$ and $\d m_{a,b}(r,\psi)$ is a probability measure given by
$$
\d m_{a,b}(r,\psi) =c_{a,b}(1-r^2)^{a-b-1}r^{2 b+1} (\sin\psi_3)^{2b-1} 
 (\sin\psi_2)^{2b-1} (\sin\psi_1)^{2b} \d r \d\psi_1\, \d\psi_2\, \d\psi_3.
$$
\end{thm}

The complication of the product formula lies in the functions $\xi_1(x,y; r, \psi_1)$ and 
$\xi_2(x,y; r,\psi)$, which are explicitly given by fairly involved formulas, and they satisfy
$\xi = (\xi_1, \xi_2)\in \UU$. Since we do not need their explicit formulas, we will not state them 
here but refer to \cite{Ko97}. 
 
For $f \in L^2(\UU; U_{a,b})$, its Fourier orthogonal series is defined by
$$
  f = \sum_{n=0}^\infty  \proj_n(U_{a,b}; f) \quad\hbox{with} \quad  
      \proj_n(U_{a,b}; f) = \sum_{k=0}^n \frac{\la f, P_{k,n}^{a,b}\ra_\UU}{h_{k,n}^{a,b}} P_{k,n}^{a,b}.
$$
The operator $\proj_n(U_{a,b}): L^2(\UU; U_{a,b}) \mapsto \CV_n(\UU,U_{a,b})$ is the orthogonal 
projection operator, which can be written as an integral 
$$
\proj_n(U_{a,b}; f, x) = \db_{a,b}\int_{\UU} f(y) \Pb_n(U_{a,b}; x,y) U_{a,b}(y) \d y
$$
in terms of the reproducing kernel $\Pb_n(U_{a,b})$ of $\CV_n(\UU,U_{a,b})$. The kernel is uniquely 
determined and it satisfies, in terms of the orthogonal basis \eqref{eq:U-basis},  
$$
\Pb_n(U_{a,b}; x,y) = \sum_{k=0}^n \frac{P_{k,n}^{a,b}(x) P_{k,n}^{a,b} (y)} {h_{k,n}^{a,b}}.
$$

The product formula \eqref{eq:prod-formulaUU} leads to a convolution structure, which can be defined as follows. 
Let $\xi_1(x,y; r, \psi_1)$ and $\xi_2(x,y; r,\psi)$ be given in \eqref{eq:prod-formulaUU}. 
For $g \in C(\UU)$, define 
$$
  (\CT_x g)(y) := \int_{[0,1]\times [0,\pi]^3} g\big(\xi_1(x,y; r, \psi_1), \xi_2(x,y; r,\psi)\big)\d m_{a,b}(r,\psi).
$$
The operator $\CT_x$ is an analogue of the translation operator, since the product formula can be written
as 
\begin{equation} \label{eq:prodUab}
 P_{k,n}^{a,b}(x) P_{k,n}^{a,b} (y) = P_{k,n}^{a,b}(\one) \big(\CT_x P_{k,n}^{a,b}\big) (y), \qquad \one = (1,1). 
\end{equation}
This generalized translation operator is bounded in the space $L^p(\UU; U_{\a,b})$. 

\begin{lem}
Let $g \in L^p(\UU;U_{a,b})$, $1 \le p < \infty$, or $g\in C(\UU)$, $p = \infty$. Then, for $x \in \UU$, 
\begin{equation}\label{eq:gen-transU}
  \|\CT_x g \|_{L^p(\UU; U_{a,b})} \le \|g\|_{L^p(\UU;U_{a,b})}, \qquad 1 \le p \le \infty,
\end{equation}
where the norm is taken as the uniform norm on $\UU$ when $p = \infty$. 
\end{lem}

This is stated and used in \cite{CFX}. Its proof follows from the product formula. Indeed, for $p = 1$, we 
use $|\CT_x g(y)| \le \CT_x (|g|)(y)$ and expand $|g|$ in its Fourier orthogonal series over $\UU$, then
we use the product formula on the Fourier series of $\CT_x (|g|)$ and integrate it to obtain, by orthogonality,
the identity $\|\CT_x |g| \|_{L^1(\UU; U_{a,b})} = \|g\|_{L^1(\UU;U_{a,b})}$. Thus, \eqref{eq:gen-transU} 
holds for $p=1$. The inequality is also trivial for $p = \infty$. The case $1< p < \infty$ then follows from 
the Riesz-Thorin theorem. 

The boundedness of $\CT_x g$ can be used to study the convergence of the Fourier orthogonal 
expansion on $\UU$. For $p \ne 2$, we need to consider a summability method. We choose the
Ces\`aro $(C,\delta)$ means, which can be given in terms of $\proj_n(U_{a,b};f)$ or the partial sum
operator $S_n(U_{a,b};f) = \sum_{k=0}^n \proj_k(U_{a,b};f)$. For the reason that will become clear later, 
we choose the latter one. The operator $S_n(U_{a,b};f)$ has the kernel 
$$
\Kb_n(U_{a,b};x,y) = \sum_{k=0}^n \Pb_k(U_{a,b};x,y).
$$ 
For $\delta > 0$, let $\Kb_n^\delta \big(U_{a,b}\big)$ denote the kernel for the Ces\`aro $(C,\delta)$ 
means, which can be written in terms of $ \Pb_k(U_{a,b})$ or $\Kb_m \big(U_{a,b}\big)$. In particular, 
\begin{equation} \label{eq:C-KernelU}
   \Kb_n^\delta (U_{a,b};x,y) = \frac{1}{\binom{n+\delta}{n}} \sum_{m=0}^n \binom{n-m+\delta-1}{n-m} 
       \Kb_m(U_{a,b}; x,y).
\end{equation}
Then the $(C,\delta)$ means $S_n^\delta (U_{a,b}; f)$ of the Fourier orthogonal series satisfy
\begin{equation*} 
   S_n^\delta (U_{a,b}; f) = \db_{a,b} \int_{\UU} f(y) \Kb_n^\delta (U_{a,b}; x,y) U_{a,b}(y) \d y. 
\end{equation*}
Since this is a linear integration operator, a standard argument shows that $S_n^\delta f$ converges
to $f$ in $L^1(\UU, U_{a,b})$ norm or in the uniform norm if and only if 
\begin{equation} \label{eq:KernelL1bdd}
   \max_{x \in \UU} \big \| \Kb_n^\delta \big(U_{a,b}; x, \cdot\big) \big \|_{L^1 (\UU;U_{a,b})} = \max_{x\in \UU}
      \db_{a,b}\int_{\UU}  \big|\Kb_n^\delta \big(U_{a,b}; x, y \big)\big| U_{a,b}(y) \d y < \infty
\end{equation}
uniformly in $n$. Now, by the definition of $\CT_x$, we have 
$$
   \Kb_n^\delta (U_{a,b}; x,y) = \CT_x \Kb_n^\delta (U_{a,b}; \one, \cdot)(y), 
$$
so that it follows from the inequality \eqref{eq:gen-transU} that 
\begin{equation} \label{eq:CesaroUat1}
    \max_{x \in \UU} \big \| \Kb_n^\delta \big(U_{a,b}; x, \cdot\big) \big \|_{L^1 (\UU;U_{a,b})} 
      \le  \big \| \Kb_n^\delta \big(U_{a,b}; \one, \cdot\big) \big \|_{L^1 (\UU;U_{a,b})}. 
\end{equation}
In particular, this shows that the convergence of $S_n^\delta(U_{a,b}; f)$ follows from the convergence
at the point $x = \one = (1,1)$. 

More generally, we could define a convolution structure for $f, g  \in L^1(\UU; U_{a,b})$ by
$$
   f \ast_U g(x):= \int_\UU f(y) (\CT_x g)(y) U_{a,b} (y) \d y, \qquad x \in \UU. 
$$
Since $\CT_x g(y)$ is symmetric in $x$ and $y$ by \eqref{eq:prodUab}, it is not difficult to see that this 
convolution is associative and commutative, as can be seen by by considering, for example, polynomials first.
Furthermore, using \eqref{eq:gen-transU}, 
it follows readily that, for $f \in L^p(\UU; U_{a,b})$ and $g \in L^1(\UU; U_{a,b})$, 
\begin{equation}\label{eq:convoluU}
      \|f\ast_U g\|_{L^p(\UU;U_{a,b})}\leq\|f\|_{L^p(\UU;U{a,b})}\, \|g\|_{L^1(\UU;U_{a,b})},\quad 1\leq p\leq\infty.
\end{equation}
The projection operators $\proj_n(U_{a,b}; f)$ and the Ces\`aro means can be written as 
\begin{equation} \label{eq:convoluSn}
\proj_n (U_{a,b}; f) =  f \ast_U \Pb_n\big(U_{a,b}; \one, \cdot\big), \qquad 
   S_n^\delta (U_{a,b}; f) =  f \ast_U \Kb_n^\delta \big(U_{a,b}; \one, \cdot\big),
\end{equation}
so that it again follows, by \eqref{eq:convoluU}, that the convergence of $S_n^\delta(U_{a,b}; f)$ reduces
to the boundedness of $\big \| \Kb_n^\delta \big(U_{a,b}; \one, \cdot\big) \big \|_{L^1 (\UU;U_{a,b})}$ of a 
single point.

It turns out that the kernel $\Kb_n (U_{a,b}; \one, \cdot)$ satisfies a closed formula \cite{CFX}. 

\begin{thm} \label{thm:Kn(1,x)}
For $a > -1$, $b > -\f12$ and $x \in \UU$, 
\begin{align} \label{eq:Kn(1,x)}
\Kb_n(U_{a,b};\one, x) = &\, \frac{P_n^{(a+b+1,b)}(1)}{h_n^{(a+b+1,b)}} 
   c_{\a+b+1,b}\int_{-1}^1 P_n^{(a+b+1,b)}(z(x,t))w_{a+b+1,b}(t)\ \d t, 
\end{align}
where  
$$
  z(x,t) =  1- (1-t^2)(1-x_1) - \frac{1}{2}(1-t)^2 (1-x_2).
$$
\end{thm}

In particular, let us denote by $k_n^\delta(w_{\a,\b}; \cdot,\cdot)$ the kernels of the Ces\`aro means of the Jacobi 
polynomials that are given by 
$$
  k_n^\delta (w_{\a,\b}; s,t) = \frac{1}{\binom{n+\delta}{n}} \sum_{k=0}^n \binom{n-k+\delta}{n-k} 
        \frac{P_k^{(\a,\b)}(s)P_k^{(\a,\b)}(t)}{h_k^{(\a,\b)}},
$$
which are the kernels of the Ces\`aro $(C,\delta)$ means of the Fourier-Jacobi series. Then \eqref{eq:Kn(1,x)}
leads to 
\begin{equation}\label{eq:KnD(1,x)}
\Kb_n^\delta (U_{a,b};\one, x) = \frac{\delta}{n+\delta} c_{\a+b+1,b}\int_{-1}^1 k_n^{\delta-1}\big(w_{a+b+1,b}; 1, z(x,t)\big)
   w_{a+b+1,b}(t)\ \d t. 
\end{equation}

The identity \eqref{eq:KnD(1,x)} allows us to bound the $L^1$ norm of $\Kb_n^\delta (U_{a,b};\one, x)$ and,
as a result, obtain the convergence of  the $(C,\delta)$ means $S_n^\delta(U_{a,b}; f)$. The result is the
following theorem established in \cite{CFX}. 
 
\begin{thm}\label{thm:CesaroU}
Let $a \ge b \ge 0$. Then the Ces\`aro means of the Fourier orthogonal series with
respect to $U_{a,b}$ satisfy 
\begin{enumerate}[ 1.]
\item if $\delta \ge a + 2 b + 4$, then $S_n^\delta (U_{a,b}; f)$ is nonnegative if $f$ is nonnegative; 
\item if $\delta > a+ b+ \f32$, then $S_n^\delta (U_{a,b}; f)$ converge to $f$ in $L^p(\UU;U_{a,b})$, $1 \le p < \infty$,
and in $C(\UU)$.   
\end{enumerate}
\end{thm}

\begin{rem} \label{rem:CesaroU}
The proof in \cite{CFX} shows the boundedness of the $L^1$ norm of $\Kb_n^\delta (U_{a,b};\one, x)$
for $a > -1$ and $b > -\f12$, so that the convergence of $S_n^\delta (U_{a,b}; f)$ at $x = \one$ holds
for $\delta > a+ b+\f32$ without the restriction $a \ge b \ge 0$. The latter condition is imposed because 
of the product formula. 
\end{rem}

The method that we outlined above also applies to other summability methods of the Fourier orthogonal series 
on $\UU$. In particular, it shows that the convergence can often be reduced to that at the point $x = \one$. 

\section{Orthogonality and Fourier orthogonal series on the surface of paraboloid}
\setcounter{equation}{0}

We consider orthogonal structure on the surface of the paraboloid of revolution
$$
\VV_0^{d+1}: =\left \{(x,t): \|x\|^2 = t, \,\, 0 \le t \le 1, \,\, x \in \RR^d \right \},
$$ 
which is compact since its $t$ direction is bounded by $0\le t \le 1$. For $x = \sqrt{t} \xi$, $\xi \in \sph$, we 
define the measure $\d \s(x,t)$ on the surface $\VV_0^{d+1}$ by $\d \s(x,t) = t^{\f{d-1}{2}} \d \s_\SS(\xi) \d t$, 
where $\d \s_\SS$ denote the surface measure of the unit sphere $\sph$. For $d\ge 2$, $\b >  - \f{d+1}{2}$ 
and $\g > -1$, we define an inner product on the surface $\VV_0^{d+1}$
$$
  \la f,g\ra_{\b,\g} = \sb_{\b,\g} \int_{\VV_0^{d+1}} f(x,t) g(x,t) \varpi_{\b,\g}(t) \d\s(x,t),
$$
where the weight function is defined by
$$
 \varpi_{\b,\g}(t) := t^\b (1-t)^\g,  \qquad 0 \le t \le 1,   
$$
and $\sb_{\b,\g}$ is the normalization constant given by 
$$
    \sb_{\b,\g} = \frac{1}{\o_d} \frac{1}{\int_0^1 t^{\f{d-1}{2}} \varpi_{\b,\g}(t) dt} =  \frac{1}{\o_d} c_{\b+\f{d-1}{2},\g}, 
$$
where the constant $c_{a,b}$ is defined in \eqref{eq:JacobiNorm} and $\o_d = 2 \pi^{\f{d}{2}}/\Gamma(\f{d}{2})$ 
is the surface area of $\sph$. The value of $\sb_{\b,\g}$ can be verified by the decomposition of the integral on 
the surface of the paraboloid
$$
  \int_{\VV_0^{d+1} } f(x,t) \d \s(x,t)  = \int_0^1 \int_{\|x\|^2 = t} f(x,t) \d\s(x,t) = 
     \int_0^1 t^{\frac{d-1}{2}} \int_{\sph} f(\sqrt{t}\xi, t) \d\s_{\SS}(\xi)\d t.
$$

\subsection{Spherical harmonics}
In order to understand the orthogonal structure on the surface of the paraboloid, we first review orthogonal
polynomials on the unit sphere which are spherical harmonics. A harmonic polynomial of degree $n$ is a 
homogeneous polynomial of degree $n$ that satisfies $\Delta Y =0$, where $\Delta$ is the Laplace operator. 
Its restriction on the unit sphere is called a spherical harmonic. Let $\CH_n^d$ denote the space of spherical 
harmonics of degree $n$ in $d$ variables. It is known that 
\begin{equation*} 
 \dim \CH_n^d =\binom{n+d-1}{n} - \binom{n+d-3}{n-2}. 
\end{equation*}   
Spherical harmonics of different degrees are orthogonal on the sphere. For $n \in \NN_0$ let
$\{Y_\ell^n: 1 \le \ell \le \dim \CH_n^d\}$ be an orthonormal basis of $\CH_n^d$ in this subsection; then 
$$
   \frac{1}{\o_d} \int_\sph Y_\ell^n (\xi) Y_{\ell'}^m (\xi)\d\s_\SS(\xi) = \delta_{\ell,\ell'} \delta_{m,n}.
$$
A fundamental property of the spherical harmonics is that they are eigenfunctions of the Laplace-Beltrami 
operator $\Delta_0$ \cite[(1.4.9)]{DaiX}, 
\begin{equation} \label{eq:sph-harmonics}
     \Delta_0 Y = -n(n+d-2) Y, \qquad Y \in \CH_n^d,
\end{equation}
where $\Delta_0$ is the restriction of $\Delta$ on the unit sphere. Another important property is that
they satisfy an addition formula \cite[(1.2.3) and (1.2.7)]{DaiX}: for $\xi \in\sph$ and $\eta \in \sph$, 
\begin{equation} \label{eq:additionF}
   \sum_{\ell =1}^{\dim \CH_n^d} Y_\ell^n (\xi) Y_\ell^n(\eta) = Z_n^{\f{d-2}{2}} (\la \xi,\eta\ra), \qquad 
    Z_n^\l(t) = \frac{n+\l}{\l} C_n^\l(t),
\end{equation}
where $C_n^\l$ is the Gegenbauer polynomial of degree $n$ that satisfies 
$$
   c_\l \int_{-1}^1 C_n^\l(t) C_m^\l (t) (1-t^2)^{\l-\f12}\d t = \frac{\l}{n+\l} C_n^\l(1) \delta_{m,n},
$$ 
where $C_n^\l(1) = (2\l)_n /n!$ and $c_\l$ is the constant determined by
\begin{equation} \label{eq:c-lam}
   c_\l = \Big(\int_{-1}^1 (1-t^2)^{\l-\f12} \d t \Big)^{-1} = \frac{\Gamma(\l+1)}{\Gamma(\f12)\Gamma(\l+\f12)}.
\end{equation}
The left-hand side of \eqref{eq:additionF} is the reproducing kernel of $\CH_n^d$ and the kernel of 
the projection operator $\proj_n: L^2(\sph) \to \CH_n^d$: 
$$
   \proj_n f(\xi)  = \frac{1}{\o_d}\int_{\sph} f(y) \sP_n(\xi,\eta) \d\s_\SS(\eta), \qquad   
       \sP_n (\xi,\eta) =  \sum_{\ell =1}^{\dim \CH_n^d} Y_\ell^n (\xi) Y_\ell^n(\eta).
$$
Thus, the product formula shows that $\proj_n f$, hence the Fourier series on the sphere defined by
$$
   L^2(\sph) = \bigoplus_{n=0}^\infty \CH_n^d\, : \qquad f = \sum_{n=0}^\infty \proj_n f,
$$ 
has a one-dimensional structure that can be used to reduce a large portion of the Fourier analysis on 
the sphere to that of the Fourier Gegenbauer series (e.g. \cite{DaiX}). 
 
 \subsection{Orthogonal structure on the surface of the paraboloid}
For $n =0,1,2,\ldots,$ let $\CV_n(\VV_0^{d+1}, \varpi_{\b,\g})$ be the space of orthogonal polynomials of 
degree $n$ with respect to the inner product $ \la \cdot, \cdot\ra_{\b,\g}$ on the surface $\VV_0^{d+1}$. 
If $Q \in \CV_n(\VV_0^{d+1}, \varpi_{\b,\g})$, then $Q(x,t)$ is a polynomial of total degree $n$ in $(x,t)$ 
variables restricted on the surface of paraboloid. Since $\VV_0^{d+1}$ is a quadratic surface, 
the dimension of $\CV_n(\VV_0^{d+1}, \varpi_{\b,\g})$ is equal to that of $\CH_n^{d+1}$ as 
established in \cite[Cor. 4.2]{OX1}; that is,
$$
  \dim \CV_n(\VV_0^{d+1}, \varpi_{\b,\g}) = \binom{n+d}{n} - \binom{n+d-2}{n-2}.
$$
An orthogonal basis of $\CV_n(\VV_0^{d+1}, \varpi_{\b,\g})$ can be given in terms of the Jacobi polynomials 
and spherical harmonics  \cite{OX1}. We also need norms of the elements in this basis. 

\begin{prop}
Let $\b > - \f{d+1}{2}$ and $\g > -1$. Let $\{Y_\ell^m: 1 \le \ell \le \dim \CH_m^d\}$ be an orthonormal
basis of $\CH_m^d$. For $0 \le m \le n$, define  
\begin{equation} \label{eq:OPbasis}
   \sQ_{m,\ell}^n(x,t) = P_{n-m}^{(\b+m +\frac{d-1}{2},\g)}(1-2 t) t^{\f{m}{2}}Y_\ell^m\left(\f{x}{\sqrt{t}}\right). 
\end{equation}
Then $\{\sQ_{m,\ell}^n: 0\le m \le n, \, 1 \le \ell \le \dim \CH_m^d\}$ is an orthogonal basis of 
$\CV_n(\VV_0^{d+1}, \varpi_{\b,\g})$.  Moreover, the norm square of $\sQ_{m,\ell}^n$ is given by 
\begin{equation} \label{eq:OPbasisNorm}
   \sh_{m,n}^{\b,\g} = \la \sQ_{m,\ell}^n,\sQ_{m,\ell}^n \ra_{\b,\g}  =
      \frac{c_{\b+\f{d-1}{2},\g}}{c_{m+\b+\f{d-1}{2},\g} } h_{n-m}^{(m+\b+\f{d-1}{2},\g)}.
\end{equation}
\end{prop}

\begin{proof}
A simple combinatorial identity shows that the cardinality of $\{(m,\ell): 0\le m \le n, \, 1 \le \ell \le \dim \CH_m^d\}$ 
is equal to $\dim \CV_n(\VV_0^{d+1}, \varpi_{\b,\g})$, so we only need to verify the orthogonality.
Since $Y_\ell^m$ is homogeneous of degree $m$, the polynomial $\sQ_{m,\ell}^n$ is of 
total degree $n$ in $(x,t)$ variables.  
Let $\a = \b+\f{d-1}{2}$. Setting $\xi = x/\sqrt{t}$, we obtain 
\begin{align*}
    \sb_{\b,\g} & \int_{\VV_0^{d+1} }  \sQ_{m,\ell}^n(x,t)  \sQ_{m',\ell'}^{n'}(x,t) \varpi_{\b,\g}(t) \d \s(x,t)  
     = \frac{1}{\o_d} \int_{\sph} Y_\ell^m(\xi) Y_{\ell'}^{m'} (\xi) \d\s_\SS(\xi) \\ 
  &  \times c_{\a,\g}\int_0^1 P_{n-m}^{(m+\a,\g)}(1-2 t)P_{n'-m'}^{(m'+\a,\g)}(1-2 t) t^{\frac{m+m'}{2}+\a} (1-t)^\g \d t.
\end{align*}
Since $Y_\ell^m$ are orthonormal, the orthogonality of $\sQ_{m,\ell}^n$ follows from that of the Jacobi polynomials
in the identity, and so is the formula for $\sh_{m,n}^{\b,\g}$. 
\end{proof}

In analogue of spherical harmonics, we have shown in \cite{X19, X20} that on the surface of the cone and the 
hyperboloid, orthogonal polynomials of degree $n$ or those that are even in $t$ variable of degree $n$ are 
eigenfunctions of a second order differential operator with eigenvalues depending only on $n$. This important 
property, however, does not hold for the paraboloid. The best we can do is the following proposition.

\begin{prop}\label{prp:pde}
Let $\b = -\f12$ and $\g > -1$. Then $\sQ_{m,\ell}^n$ in \eqref{eq:OPbasis} satisfies the differential equation
\begin{align}\label{eq:diff-eqn}
 & \left[   t (1-t) \partial_t^2 +\left(\tfrac{d}{2} - (\g + \tfrac{d}2 + 1)t \right) \partial_t + \frac{1-t}{4 t} \Delta_0^{(\xi)}
    \right] y \\
 &\qquad \qquad \qquad \qquad \qquad = -  \left(n \big( n+\g + \tfrac{d}{2}\big) - m \big(n + \tfrac{\g + d-1}{2}\big) \right)y,  
 \notag
\end{align}
where $\Delta_0^{(\xi)}$ is the Laplace-Beltrami operator acting on $\xi = x/\sqrt{t} \in \sph$.  
\end{prop}

\begin{proof}
Let $f_{n,m}(t) = P_{n-m}^{(m+\a,\g)}(1-2t) t^{\f{m}{2}}$ and $\a = \b + \frac{d-1}{2}$, so that 
$\sQ_{m,\ell}^n(x,t) = f_{n,m}(t) Y_\ell^m(\xi)$, where $\xi = x/\sqrt{t}$. Since the Jacobi polynomial 
$P_n^{(\a,\g)}(1-2t)$ satisfies the differential equation
\begin{equation}\label{eq:Jacobi-diff}
 t(1-t)y'' + (1+\a - (\a+\g+2) t) y' + n (n+\a+\g+1) y =0, 
\end{equation}
a straightforward computation shows that $f_{n,m}$ satisfies
\begin{align*}
t (1-t) f_{n,m}''(t) +& \big(1+\a - (2+\a+\g)t\big) f_{n,m}'(t) -m (m  + 2\a) \frac{1-t}{4t} f_{n,m}(t) \\
   & = - \big(n (n+ \a + \g + 1) - \tfrac12 m (2n +  2 \a + \g)\big) f_{n,m}(t).
\end{align*}
For $\b = -\f12$, we have $2\a = d-2$, so that $- m(m+2\a)$ is the eigenvalue of the Laplace-Beltrami 
operator $\Delta_0^{(\xi)}$. Hence, multiplying the above identity by $Y_{\ell}^m(\xi)$ and replacing
$-m (m  + 2\a) Y_\ell^m$ by $\Delta_0^{(\xi)} Y_\ell^m$, we have proved \eqref{eq:diff-eqn}.
\end{proof}

\begin{rem}
The right-hand side of \eqref{eq:diff-eqn} depends on $m$, so that $\sQ_{m,\ell}^n$ is the eigenfunction of
the operator in the left-hand side of \eqref{eq:diff-eqn} with the eigenvalue depends on both $m$ and $n$. 
This is in sharp contrast with the orthogonal structure on the cone and on the hyperboloid, for which the 
eigenvalue depends only on the degree of the orthogonal polynomials. 
\end{rem}

The reproducing kernel of $\CV_n(\VV_0^{d+1}, \varpi_{\b,\g})$ is defined by 
\begin{align*}
  \sP_n\left(\varpi_{\b,\g}; (x,t),(y,s)\right) = \sum_{m=0}^n \sum_{\ell=1}^{\dim\CH_{m}^d}
        \frac{ \sQ_{m,\ell}^n(x,t)  \sQ_{m,\ell}^n(y,s)}{\sh_{m,n}^{\b,\g}},
\end{align*}   
which is the kernel of the orthogonal projection operator $\proj_n(\varpi_{\b,\g})$ from the space
$L^2(\VV_0^{d+1};\varpi_{\b,\g})$ onto $\CV_n(\VV_0^{d+1};\varpi_{\b,\g})$. In contrast to the cone and to
the hyperboloid, this kernel does not satisfy an addition formula that is of a closed formula of one-dimensional in
essence. Instead, however, we can express this kernel in terms of the reproducing kernel 
$\Pb_n(U_{a,b}; \cdot,\cdot)$ on the parabola domain $\UU$. 

\begin{thm}
Let $d \ge 2$ and $\g> -1$. Let $(x,t) = (\sqrt{t} \xi,t) \in \VV_0^{d+1}$ and $(y,s) = (\sqrt{s} \eta, s) \in \VV_0^{d+1}$ 
with $\xi, \eta \in \sph$. Then for $\b > -\f12$,
\begin{align} \label{eq:reprodV1}
  \sP_n & \big(\varpi_{\b,\g}; (x,t),(y,s)\big)  = 
    c_{ \frac{d-2}{2},\b-\f12} c_{\b+\f12}  \\
  & \qquad \times \int_{-1}^1 \int_{-1}^1  \Pb_n \left(U_{\g,\b+\frac{d-1}{2}}; 
          \big(\sqrt{t}, t \big),\big(\sqrt{s}\big ( \tfrac{1-z_1}{2}  \la \xi,\eta \ra + \tfrac{1+z_1}2 z_2\big), s \big)\right) \notag \\
  & \qquad \times (1-z_1)^\frac{d-2}{2} (1+z_1)^{\b-\f12}  (1-z_2^2)^{\b}  \d z_1 \d z_2, \notag
\end{align}
and for $\b = - \f12$, 
\begin{align} \label{eq:reprodV2}
 \sP_n(\varpi_{-\f12,\g}; (x,t),(y,s)) =  
        \Pb_n \left(U_{\g,\frac{d-2}{2}}; \big(\sqrt{t},t\big),\big( \la \xi, y\ra, s \big)\right).
\end{align}
\end{thm}

\begin{proof}
First we note that the right-hand side of \eqref{eq:reprodV1} is well defined. Indeed, since $|\la \xi,\eta\ra|\le 1$
and $-1\le z_1, z_2 \le 1$, we see $|\sqrt{s}\big ( \tfrac{1-z_1}{2}  \la \xi,\eta \ra + \tfrac{1+z_1}2 z_2\big)| 
\le \sqrt{s}$, so that both variables in $\Pb_n( U_{\g,\b+\frac{d-1}{2}})$ are elements of $\UU$. 

We first need to specify the reproducing kernel $\Pb_n(U_{a,b})$ on $\UU$ when one of its variables is 
on the curved boundary of $\UU$. Using the well-known identity \cite[(4.7.1)]{Sz}
$$
  \frac{P_m^{(b-\f12,b-\f12)}(1) P_m^{(b-\f12,b-\f12)}(\rho)} {h_m^{(b-\f12,b-\f12)}}
    = \frac{C_m^b (1) C_m^b(\rho)} {h_m^b} = Z_m^b (\rho),
$$
it follows from \eqref{eq:U-basis} and \eqref{eq:U-basisNorm} that 
\begin{align} \label{eq:PnUab}
 &  \Pb_n \left(U_{a,b}; (\sqrt{x_2},x_2), (y_1,y_2)\right) = \sum_{m=0}^n \frac{P_{m,n}^{a,b}(\sqrt{x_2},x_2)
      P_{m,n}^{a,b}(y_1,y_2)}{h_{m,n}^{a,b}}\\
 &  \qquad     = \sum_{m=0}^n \frac{c_{b+m,a}}{c_{b,a}}
          \frac{P_{n-m}^{(b+m,a)}(1-2x_2)P_{n-m}^{(b+m,a)}(1-2y_2)}{h_{n-m}^{(b+m,a)}} x_2^{\f{m}{2}}y_2^{\f{m}{2}}
          Z_m^b \left(\frac{y_1}{\sqrt{y_2}}\right). \notag
\end{align}
In terms of the orthogonal basis \eqref{eq:OPbasis} and using the addition formula \eqref{eq:additionF} for the 
spherical harmonics, we obtain, with $\a = \b+\frac{d-1}{2}$, 
\begin{align*}
  \sP_n(\varpi_{\b,\g}; (x,t),(y,s)) = \sum_{m=0}^n  
            \frac{P_{n-m}^{(\a+m,\g)}(1-2 t)P_{n-m}^{(\a+m,\g)}(1-2 s)} {\sh_{m,n}^{\b,\g}}
               t^{\f{m}{2}} s^{\f{m}{2}} Z_m^{\f{d-2}2}(\la \xi,\eta\ra),
\end{align*}   
where $\xi = \frac{x}{\sqrt{t}} \in \sph$ and $\eta=  \frac{y}{\sqrt{s}} \in \sph$. For $\b = -\f12$, the sum in 
the right-hand side can be identified, using \eqref{eq:OPbasisNorm}, with $\Pb_n(U_{\g,\f{d-2}{2}})$ in 
\eqref{eq:PnUab} by setting $x_2 = t$, $y_2 = s$ and $y_1 = \sqrt{s} \la \xi,\eta\ra$, which proves \eqref{eq:reprodV2}. 

For $\b > - \f12$, we need to increase the value of the index in the zonal harmonic by using the following identity, 
proved recently in \cite{X15}, 
\begin{align}\label{eq:ZtoZ}
  Z_m^\l(t) =  c_{\l,\s-1} c_{\s}  \int_{-1}^1 \int_{-1}^1 & Z_m^{\l+\s} \left( \tfrac{1-z_1}{2} t + \tfrac{1+z_1}2 z_2\right)\\
      & \times (1-z_1)^\l (1+z_1)^{\s-1}  (1-z_2^2)^{\s-\f12}  \d z_1 \d z_2 \notag
\end{align}
with $\l = \frac{d-2}{2}$ and $\s = \b+\f12$. This shows that
\begin{align*}
 \sP_n \big(\varpi_{\b,\g};&  (x,t),  (y,s)\big)  = 
    c_{\frac{d-2}{2},\b-\f12} c_{\b+\f12}   \int_{-1}^1 \int_{-1}^1   \sum_{m=0}^n  
            \frac{P_{n-m}^{(\a+m,\g)}(1-2 t)P_{n-m}^{(\a+m,\g)}(1-2 s)} {\sh_{m,n}^{\b,\g}}\\
      & \times  t^{\f{m}{2}} s^{\f{m}{2}} Z_m^{\a}\big( \tfrac{1-z_1}{2} \la \xi,\eta\ra + \tfrac{1+z_1}2 z_2\big)
       (1-z_1)^\frac{d-2}{2} (1+z_1)^{\b-\f12}  (1-z_2^2)^{\b}  \d z_1 \d z_2.
\end{align*}
The sum in the right-hand side can be identified, using \eqref{eq:OPbasisNorm}, with the reproducing kernel 
$\Pb_n(U_{\g, \b+\frac{d-1}{2}})$ in \eqref{eq:PnUab}, which gives \eqref{eq:reprodV1}. 
\end{proof}

The identity \eqref{eq:reprodV2} for $\b = -\f12$ can be regarded as the limit of \eqref{eq:reprodV1} 
as $\b \to -\f12$ by using the limit relation \eqref{eq:limit-int}.

The kernel $\sP_n(\varpi_{\b,\g})$, however, does not satisfy a closed formula in general. In the case that  
$(x,t)$ is on the boundary $t=1$ of the paraboloid $\VV_0^{d+1}$, however, we could derive a closed
form formula for the kernel of the partial sum operator 
$$
\sK_n\big(\varpi_{\b,\g}, (x,t), (y,s)\big) = \sum_{m=0}^n \sP_m\big(\varpi_{\b,\g}; (x,t), (y,s)\big)
$$
on the paraboloid by using the closed formula of $\Kb_n(U_{a,b}; \one, \cdot)$ in Theorem \ref{thm:Kn(1,x)}.
We state the result for $\b = -\f12$ as an example. 

\begin{cor}
Let $d \ge 2$ and $\g > -1$. Then, for $\xi \in \sph$, 
\begin{align*}
\sK_n \big(\varpi_{-\f12,\g};  (\xi,1), (y,s) \big) &=  \frac{P_n^{(\g+ \f{d}2,\f{d-2}{2})}(1)}{h_n^{(\g+ \f{d}2,\f{d-2}{2})}} \\
  & \times c_{\g+ \f{d}2,\f{d-2}{2}} \int_{-1}^1 P_n^{(\g+ \f{d}2,\f{d-2}{2})}\big(z'(\xi, y, v)\big)
     w_{\g+ \f{d}2,\f{d-2}{2}}(v)\ \d v, 
\end{align*}
where, with $y = \sqrt{s} \eta$, 
$$
 z'(\xi,y, v) = z( (\la \xi, \eta\ra, s), v) = 1 - (1-v^2) (1- \la \xi,y\ra) - \frac12 (1-v)^2 (1- s).
$$
\end{cor}
 
This is a corollary of \eqref{eq:reprodV2} and \eqref{eq:Kn(1,x)}. Similarly, using \eqref{eq:reprodV1} 
and \eqref{eq:Kn(1,x)}, we can derive an explicit formula for $\b > -\f12$, which is however more involved.

\subsection{Summability of Fourier orthogonal series}
Let $ \proj_n ( \varpi_{\b,\g})$ be the orthogonal projection operator 
$$
\proj_n( \varpi_{\b,\g}): L^2(\VV_0^{d+1}; \varpi_{\b,\g})
   \mapsto \CV_n(\VV_0^{d+1}, \varpi_{\b,\g}).
$$
In terms of the reproducing kernel  $\sP_n(\varpi_{\b,\g})$ of $\CV_n(\VV_0^{d+1}, \varpi_{\b,\g})$, we have
$$
 \proj_n ( \varpi_{\b,\g}; f) = \sb_{\b,\g} \int_{\VV_0^{d+1}} f(y,s) \sP_n\big(\varpi_{\b,\g}; (x,t), (y,s)\big) 
   \varpi_{\b,\g}(s) \d \s(y,s).
$$
For $f\in L^2(\VV_0^{d+1}; \varpi_{\b,\g})$, the Fourier orthogonal series of $f$ on $\VV_0^{d+1}$ is defined by
$$
   f = \sum_{n=0}^\infty \proj_n ( \varpi_{\b,\g}; f).
$$

Below we study the summability of this Fourier orthogonal series. We start with a definition. 
Recall that $\UU$ denotes the domain bounded by the parabolic $x_2 = x_1^2$ and $x_2=1$ in $\RR^2$. 
We further denote by $\UU_0$ the curved portion of the boundary of $\UU$,
$$
  \UU_0:= \{(x_1,x_2) \in \UU: x_2 = x_1^2\}.
$$

\begin{defn} \label{def:Tbg-V0}
Let $\b \ge -\f12$ and $\g > -1$. Let $g: \UU_0\times \UU \mapsto \RR$ such that, for each $t \in [0,1]$, 
the function $x \mapsto g((\sqrt{t},t); x)$ is in $L^1\big(\UU;U_{\g, \a}\big)$. For $(x,t) = 
(\sqrt{t}\xi,t) \in \VV_0^{d+1}$ and $(y,s) = (\sqrt{s} \eta,s) \in \VV_0^{d+1}$, define for $\b > -\f12$
\begin{align*}
   \sT_{\b,\g} g\big( (x,t), (y,s)\big) :=  
        c_{ \frac{d-2}{2},\b-\f12} c_{\b+\f12}   \int_{-1}^1 \int_{-1}^1 & g\left( \big(\sqrt{t}, t \big),
         \big(\sqrt{s}\big ( \tfrac{1-z_1}{2}  \la \xi,\eta \ra + \tfrac{1+z_1}2 z_2\big), s \big)\right)  \\
  & \times (1-z_1)^\frac{d-2}{2} (1+z_1)^{\b-\f12}  (1-z_2^2)^{\b}  \d z \notag
\end{align*}
and, furthermore, define for $\b = -\f12$ 
$$
    \sT_{-\f12,\g} g\big( (x,t), (y,s)\big) :=  g\left(\big(\sqrt{t}, t \big), \big( \la \xi, y\ra, s \big) \right).
$$
\end{defn}

The definition of $\sT_{\b,\g}$ is motivated by the relation, by \eqref{eq:reprodV1} and  \eqref{eq:reprodV2}, 
\begin{equation}\label{eq:PnV=PnU}
   \sP_n (\varpi_{\b,\g}; (x,t), (y,s)) =  \sT_{\b,\g} \Pb_n (U_{\g, \b+\frac{d-1}{2}}) \big( (x,t), (y,s)\big).
\end{equation}
For each fixed $(x,t) \in \VV_0^{d+1}$, this is a bounded operator as seen below. 

\begin{prop}\label{prop:Tbg-bdd}
Let $\b \ge -\f12$ and $\g > -1$. Let $g: \UU_0\times \UU \mapsto \RR$ such that, for each $t \in [0,1]$, 
the function $g\big((\sqrt{t},t); \cdot\big)$ is in $L^1\big(\UU;U_{\g, \a}\big)$ with $\a =\b+\frac{d-1}{2}$.
Then, for $(x,t) \in \VV_0^{d+1}$, 
\begin{equation}\label{eq:Tbg-bdd}
   \int_{\VV_0^{d+1}} \left|\sT_{\b,\g} g\big((x,t), (y,s) \big) \right| \varpi_{\b,\g}(s) \d \s(y,s) 
    \le c  \int_{\UU} \left|g \left(\big(\sqrt{t},t\big), z\right)\right| U_{\g,\a}(z) \d z. 
\end{equation}
\end{prop}

\begin{proof} 
Let $G(z) = g( (t,\sqrt{t}), z)$ for $z \in \UU$ in this proof. We first consider the case $\b = -\f12$. 
Using the well-known integral relation
\begin{equation}  \label{eq:IntSph}
 \frac{1}{\o_d} \int_{\sph} f(\la \xi,\eta \ra) \d \s(\xi) =  c_{\f{d-2}{2}} \int_{-1}^1 f(u) (1-u^2)^{\f{d-3}2} \d u 
\end{equation}
we obtain, setting $y = \sqrt{s}\, \eta$ wtih $\eta \in \sph$, that
\begin{align*}
 &  \sb_{-\f12,\g} \int_{\VV^{d+1}}  \left|\sT_{-\f12,\g} g\big( (x,t), (y,s)\big)\right| \varpi_{-\f12,\g}(s) \d \s(y,s) \\
 & \qquad =   \sb_{-\f12,\g} 
   \int_0^1 s^{\frac{d-1}{2}} \int_{\sph} \left |G\big(\sqrt{s} \, \la \xi,\eta\ra, s\big)\right|\varpi_{-\f12,\g}(s) \d \s(\xi) \d s \\
 & \qquad = \sb_{-\f12,\g} c_{\f{d-2}{2}} \o_d  \int_0^1 \int_{-1}^1  \left |G\big(\sqrt{s} \, u, s\big)\right|
       (1-u^2)^{\f{d-3}{2}} \d u \,
      s^{\frac{d-2}{2}} (1-s)^\g \d s \\
  & \qquad =  \db_{\g,\f{d-2}{2}} \int_\UU  |G(z)| (z_2-z_1^2)^{\f{d-3}{2}} (1-z_2)^\g \d z,     
\end{align*}
where we have used \eqref{eq:integralBB} in the last step, and the constant can be verified simply 
by the fact that if $g = 1$, then $\sT_{\b,\g} g =1$. In particular, this shows that the inequality \eqref{eq:Tbg-bdd} 
is in fact an identity for $\b=-\f12$. 

We now consider the case $\b > -\f12$. Using \eqref{eq:IntSph}, we obtain
\begin{align*}
 \int_{\VV_0^{d+1}}&  \left | \sT_{\b,\g} g\big( (x,t), (y,s)\big)\right| \varpi_{\b,\g}(s) \d \s(y,s) \\
  &  \le    c  \int_{0}^1 \int_{-1}^1 \int_{-1}^1 \int_{-1}^1 
          \left| G\big( \sqrt{s}\big ( \tfrac{1-z_1}{2}  u + \tfrac{1+z_1}2 z_2\big), s \big)\right|\\
  &  \times (1-z_1)^\frac{d-2}{2} (1+z_1)^{\b-\f12}  (1-z_2^2)^{\b}  (1-u^2)^{\f{d-3}2}   
      s^{\b+\frac{d-1}{2}} (1-s)^\g \d z \d u\d s, 
\end{align*}
where $c = \o_d c_{\f{d-2}{2}} c_{ \frac{d-2}{2},\b-\f12} c_{\b+\f12}$. Making a change of variables $z_1 \mapsto y$
with
$$
   y = \tfrac{1-z_1}{2}  u + \tfrac{1+z_1}2 z_2
$$
and dividing the integral over $\d u \d z_2$ on $[-1,1]^2$ into two integrals over the triangles 
$\{(u, z_2) \in [-1,1]^2: u \ge z_2\}$ and $\{(u, z_2) \in [-1,1]^2: u < z_2\}$, respectively, we can write 
the triple integral against $\d u \d z$ as a sum of two integrals
\begin{align*}
& 2^{\b+\f{d-1}{2}} \int_{-1}^1 \int_{u}^1 \int_{u}^{z_2}  \left| G(\sqrt{s}y,s)\right|  (z_2-y)^{\f{d-2}{2}} (y-u)^{\b-\f12} dy 
    \frac{(1-z_2^2)^\b (1-u^2)^{\f{d-3}{2}}}{(z_2-u)^{\b+\f{d-1}{2}}} \d z_2 \d u \\
+& 2^{\b+\f{d-1}{2}} \int_{-1}^1 \int_{-1}^u \int_{z_2}^{u}  \left| G(\sqrt{s}y,s)\right|  (y-z_2)^{\f{d-2}{2}} (u-y)^{\b-\f12} dy 
    \frac{(1-z_2^2)^\b (1-u^2)^{\f{d-3}{2}}}{(u-z_2)^{\b+\f{d-1}{2}}} \d z_2 \d u.    
\end{align*}
Changing the order of integrals in both terms, we see that this sum is equal to 
\begin{align*}
  2^{\b+\f{d-1}{2}} \int_{-1}^1  \left| G(\sqrt{s}y,s) \right| & \left [ \int_{-1}^y \int_{y}^{1} (z_2-y)^{\f{d-2}{2}} (y-u)^{\b-\f12}   
    \frac{(1-z_2^2)^\b (1-u^2)^{\f{d-3}{2}}}{(z_2-u)^{\b+\f{d-1}{2}}} \d z_2 \d u \right. \\
 + & \left. \int_{y}^1 \int_{-1}^{y}  (y-z_2)^{\f{d-2}{2}} (u-y)^{\b-\f12}   
    \frac{(1-z_2^2)^\b (1-u^2)^{\f{d-3}{2}}}{(u-z_2)^{\b+\f{d-1}{2}}} \d z_2 \d u \right] \d y.    
\end{align*}
Making a change of variables $(u,z_2) \mapsto (v_1,v_2)$ with $v_1 = (z_2-y)/(1-y)$ and $v_2 = (y-u)/(1+y)$ 
in the first double integral in the square bracket, and a change of variables $(u,z_2) \mapsto (v_1,v_2)$ with 
$v_1 = (y-z_2)/(1+y)$ and $v_2 = (u-y)/(1-y)$ in the second double integral in the square bracket, we see that
the expression in the above square bracket is equal to    
\begin{align*}
 (1-y^2)^{\b+\f{d-2}{2}}  & \bigg [(1-y) \int_{0}^1 \int_{0}^{1} 
      \frac{(1+y+(1-y)v_1)^\b (1-y+(1+y)v_2)^{\f{d-3}{2}}}{((1-y)v_1+(1+y)v_2)^{\b+\f{d-1}{2}}}   \\
  & \qquad\qquad\qquad \times v_1^{\f{d-2}{2}} v_2^{\b-\f12} (1-v_1)^\b (1- v_2)^{\f{d-3}{2}}\d v_1\d v_2  \\
 + &\, (1+y) \int_{0}^1 \int_{0}^{1} 
      \frac{(1-y+(1+y)v_1)^\b (1+y+(1-y)v_2)^{\f{d-3}{2}}}{((1+y)v_1+(1-y)v_2)^{\b+\f{d-1}{2}}}   \\
  &  \qquad\qquad\qquad \times v_1^{\f{d-2}{2}} v_2^{\b-\f12} (1-v_1)^\b (1- v_2)^{\f{d-3}{2}}\d v_1\d v_2  \bigg] \d y.    
\end{align*}
Since $0 \le v_1, v_2 \le 1$ and $1 \pm y \ge 0$, it follows that $(1-y)v_1 \le (1-y)v_1+(1+y)v_2$, 
$$
 \frac{v_2 (1+y+(1-y)v_1)}{((1-y)v_1+(1+y)v_2)} \le 1\quad \hbox{and} \quad 
  \frac{v_1 (1-y+(1+y)v_2)}{((1-y)v_1+(1+y)v_2)} \le 1, 
$$
which implies that the first term in the square bracket is bounded by 
$$
   \int_0^1\int_0^1 v_1^{-\f12} v_2^{-\f12} (1-v_1)^\b (1-v_2)^\frac{d-3}{2} \d v_1 \d v_2 = \frac{\pi \Gamma(\b+1)\Gamma(\frac{d-1}{2})}{\Gamma(\b+\f32)\Gamma(\frac{d}{2})}.
$$
Similarly, it is easy to see that the same bound holds for the second term in the square bracket. 
Putting all these estimates together, we conclude that 
\begin{align*}
   \int_{\VV_0^{d+1}}&  \left |\sT_{\b,\g} g\big( (x,t), (y,s)\big)\right | \varpi_{\b,\g}(s) \d \s(y,s)\\
    & \le c \int_0^1 \int_{-1}^1 |G(\sqrt{s}y,s) | (1-y^2)^{\b+\f{d-2}{2}}  s^{\b+\frac{d-1}{2}} (1-s)^\g \d y \d s \\
    & = c   \int_{\UU} |G(z_1,z_2)| U_{\g,\b+ \frac{d-1}{2}} (z_1,z_2) \d z, 
\end{align*}
where the last step follows from \eqref{eq:integralBB}. This completes the proof. 
\end{proof}

\begin{defn}
Let $\b \ge -\f12$ and $\g > -1$.  Let $g: \UU_0\times \UU \mapsto \RR$ such that, for each $t \in [0,1]$, 
the function $x \mapsto g((\sqrt{t},t); x)$ is in $L^1\big(\UU; U_{\g, ,\b+\f{d-1}2}\big)$. 
For $f \in L^1(\VV_0^{d+1}; \varpi_{\b,\g})$ and $(y,s) \in \VV_0^{d+1}$, define
$$
  (f\ast_{\VV_0} g)(y,s) = \sb_{\b,\g}\int_{\VV_0^{d+1}} f(x,t) \sT_{\b,\g} g\big( (x,t), (y,s) \big) \varpi_{\b,\g}(t) \d\s(x,t).
$$
\end{defn}

The operator $\sT_{\b,\g}$ plays the role of a translation in the definition of the pseudo convolution 
$f\ast_{\VV_0} g$. By \eqref{eq:PnV=PnU}, it follows that the project operator onto 
$\CV_n(\VV_0^{d+1}, \varpi_{\b,\g})$
satisfies
$$
  \proj_n(\varpi_{\b,g}; f) = f \ast_{\VV_0} \Pb_n\left(U_{\g,\b+\f{d-1}2}\right).
$$

\begin{prop} \label{prop:f*gV0}
Let $\b \ge -\f12$ and $\g > -1$. For $f \in L^p (\VV_0^{d+1};\varpi_{\b,\g})$, $1 \le p < \infty$, and
$f \in C(\VV_0^{d+1})$ for $p = \infty$, 
$$
    \| f\ast_{\VV_0} g\|_{L^p (\VV_0^{d+1};\varpi_{\b,\g})} \le c \|f\|_{L^p (\VV_0^{d+1};\varpi_{\b,\g})}
      \max_{t\in [0,1]} \int_{\UU} \left|g\left( \big(\sqrt{t}, t), z \right)\right| U_{\g,\b+\f{d-1}{2}}(z) \d z.
$$
\end{prop}

\begin{proof}
By the Minkowski inequality, we obtain
$$
    \| f\ast_{\VV_0} g\|_{L^p (\VV_0^{d+1};\varpi_{\b,\g})} \le \|f\|_{L^p (\VV_0^{d+1};\varpi_{\b,\g})}
       \sb_{\b,\g} \int_{\UU} \left | T_{\b,\g} g\left( (x,t), (y,s) \right) \right| \varpi_{\b,\g}(s) \d \s (y,s).
$$
Applying the inequality \eqref{eq:Tbg-bdd} on the integral in the right-hand side, the stated inequality
follows readily by taking the maximum over $t$.
\end{proof}

The boundedness of the pseudo convolution can be used to study the convergence of the Fourier 
orthogonal series on the surface of the paraboloid. As in the case of the previous section, we consider 
the Ces\`aro means. For $\delta > -1$, let $\sK_n^\delta(\varpi_{\b,\g}; (x,t),(y,s))$ be the kernel of the Ces\`aro means \
$\sS_n^\delta \big(\varpi_{\b,\g}; f)$, which can be written in terms of the reproducing kernel 
$\sP_n(\varpi_{\b,g})$ in analogous to \eqref{eq:C-KernelU}, and it satisfies, by \eqref{eq:PnV=PnU}, that 
\begin{equation}\label{eq:KdV=KdU}
  \sK_n^\delta (\varpi_{\b,\g}; (x,t), (y,s)) =  \sT_{\b,\g} \Kb_n^\delta (U_{\g, \b+\frac{d-1}{2}}) \big( (x,t), (y,s)\big).
\end{equation}
In terms of the pseudo convolution, we can write 
$$
\sS_n^\delta \big(\varpi_{\b,\g}; f, (x,t) \big) = f \ast_{\VV_0} \Kb_n^\delta \big(U_{\g, \b+\frac{d-1}{2}}\big)(x,t).
$$

\begin{thm}
Let $d\ge 2$, $\b \ge - \f12$ and $\g> -1$. If $f \in C(\VV_0^{d+1})$, then $\sS_n^\delta\big(\varpi_{\b,\g}; f, (\xi,1)\big)$ 
converges to $f(\xi, 1)$ uniformly for $\xi \in \sph$ provided $\delta > \b+\g+ \f{d+2}{2}$.
\end{thm}

\begin{proof}
The convergence of $\sS_n^\delta\big(\varpi_{\b,\g}; f, (\xi,1)\big)$ holds if and only if 
$$
 \sup_{\xi \in \sph} \int_{\VV_0^{d+1}} \left| \sK_n^\delta (\varpi_{\b,\g}; (\xi,1), (y,s)) \right| \varpi_{\b,\g}(s) \d \s(y,s) 
$$ 
is bounded uniformly in $n$. By \eqref{eq:KdV=KdU} and the inequality \eqref{eq:Tbg-bdd}, this is bounded by 
$$
 \int_{\UU} \left| \Kb_n^\delta \big(U_{\g,\b+\f{d-1}{2}}; \one, z\big) \right| U_{\g, \b+\f{d-1}{2}}(z) \d z.
$$
For $\delta > \g+\b+\f{d-1}{2} +\f32$, the last integral is bounded uniformly in $n$ by Theorem \ref{thm:CesaroU} 
and by Remark \ref{rem:CesaroU}.
\end{proof}

\begin{thm} \label{thm:(C,d)V0}
Let $d\ge 2$, $\g \ge \b + \f{d-1}{2}$ and $\b \ge -\f12$. Let $f \in L^p(\VV_0^{d+1}, \varpi_{\b,\g})$ for 
$1 \le p < \infty$ and $f \in C(\VV_0^{d+1})$ for $p = \infty$. Then the Ces\`aro means $\sS_n^\delta\big(\varpi_{\b,\g}; f)$
satisfy 
\begin{enumerate}[ 1.]
\item if $\delta \ge 2 \b +\g+ d+3$, then $\sS_n^\delta (\varpi_{\b,g}; f)$ is nonnegative if $f$ is nonnegative; 
\item if $\delta > \b+ \g+ \f{d+2}2$, then $\sS_n^\delta (\varpi_{\b,g}; f)$ converge to $f$ in 
$L^p(\VV_0^{d+1}; \varpi_{\b,\g})$, $1 \le p < \infty$, and in $C(\VV_0^{d+1})$.   
\end{enumerate}
\end{thm}
 
\begin{proof}
The positivity of $\sS_n^\delta\big(\varpi_{\b,\g}; f,)$ follows from the positivity of its kernel. Hence, the first item
is the consequence of \eqref{eq:KdV=KdU} and the positivity of $\sT_{\b,\g} \Kb_n^\delta (U_{\g, \b+\frac{d-1}{2}})$, 
where the latter follows from the positivity of $ \Kb_n^\delta (U_{\g, \b+\frac{d-1}{2}})$, which in turn follows from
the definition of $\sT_{\b,\g}$ and Theorem \ref{thm:CesaroU}; the last theorem requires $a \ge b \ge 0$
in $U_{a,b}$, which is satisfied in our case by the assumption $\g \ge \b + \f{d-1}{2}$ and $\b \ge - \f12$.

For the second item, it suffices to show that $L^p$ norm of $\sS_n^\delta (\varpi_{\b,g}; f)$ is uniformly bounded. 
By Proposition \ref{prop:f*gV0}, it is sufficient to show that
$$
  \max_{t \in [0,1]} \int_{\UU} \left| \Kb_n^\delta \left(U_{\g, \b+\frac{d-1}{2}}; \big(\sqrt{t},t \big), z \right)\right|
     U_{\g,\b+\f{d-1}{2}}(z)\d z 
$$
is bounded uniformly in $n$ when $\delta > \g + \b +\frac{d-1}{2} +\f32$. This follows immediately from  
\eqref{eq:CesaroUat1} and Theorem \ref{thm:CesaroU}.
\end{proof}

\section{Orthogonality and Fourier orthogonal series on the solid paraboloid}
\setcounter{equation}{0}

We consider orthogonal structure on the solid paraboloid of revolution
$$
\VV^{d+1}: =\left \{(x,t): \|x\|^2 \le t, \,\, 0 \le t \le 1, \,\, x \in \RR^d \right \}, 
$$ 
which is bounded by the surface $\VV_0^{d+1}$ and the hyperplane $t=1$ of $\RR^{d+1}$.
The $t$-section of the domain, $\{x :\|x\| \le \sqrt{t}\}$, is the ball of radius $\sqrt{t}$ in $\RR^d$. W
review the orthogonal structure on the unit ball first. 

\subsection{Classical orthogonal polynomials on the unit ball} 
For $\mu > -\f12$, let $\varpi_\mu$ be the weight function 
$$
 \varpi_\mu(x):= (1-\|x\|^2)^{\mu-\f12}, \qquad  \|x\| < 1.
$$
The classical orthogonal polynomials on the unit ball are orthogonal with respect to the inner product 
$$
  \la f,g\ra_{\mu} =b_\mu \int_{\BB^d} f(x) g(x) \varpi_\mu(x) \d x \quad \hbox{with}\quad 
    b_\mu = \frac{\Gamma(\mu+\f{d+1}{2})}{\pi^{\f{d}{2}}\Gamma(\mu+\f12)},
$$
where $b_\mu$ is the normalization constant of $\varpi_\mu$ so that $\la 1,1\ra=1$. 

Let $\CV_n(\BB^d,\varpi_\mu)$ be the space of orthogonal polynomials of degree $n$ with respect to 
$\varpi_\mu$. Then $\dim \CV_n(\BB^d, \varpi_\mu) = \binom{n+d-1}{n}$.  An orthogonal 
basis of $\CV_n(\BB^d, \varpi)$ can be given in terms of the Jacobi polynomials or spherical harmonics, 
see \cite[Chapter 5]{DX}, which we shall call the {\it basis with parity}, since its elements are polynomials 
that are even in each of its variables if $n$ is even and odd in each of its variables if $n$ is odd. The 
orthogonal polynomials of degree $n$ are eigenfunctions of a second order differential operator: 
for $u \in \CV_n(\BB^d, \varpi_\mu)$,  
\begin{equation}\label{eq:diffBall}
  \left( \Delta  - \la x,\nabla \ra^2  - (2\mu+d-1) \la x ,\nabla \ra \right)u = - n(n+2\mu+ d-1) u.
\end{equation}
Furthermore, these polynomials also satisfy an addition formula. Let $\Pb_n(\varpi_\mu;\cdot,\cdot)$ be 
the reproducing kernel of the space $\CV_n(\BB^d,\varpi_\mu)$. In terms of an orthonormal basis 
$\{P_{\kb}^n: |\kb| =n\}$ of $\CV_n(\BB^d,\varpi_\mu)$, the kernel can be written as 
$$
\Pb_n(\varpi_\mu;x,y) = \sum_{|\kb| = n} P_{\kb}^n(x) P_{\kb}^n(y). 
$$
The addition formula on the unit ball states \cite{X99}, for $\mu \ge 0$, 
\begin{align}\label{eq:additionB}
 \Pb_n(\varpi_\mu;x,y) = c_{\mu-\f12}
      \int_{-1}^1 Z_n^{\mu+\f{d-1}{2}} & \left(\la x,y\ra+ t \sqrt{1-\|x\|^2}\sqrt{1-\|y\|^2} \right) \\
       &  \times  (1-t^2)^{\mu-1} \d t, \notag
\end{align}
where the identity holds for $\mu =0$ under the limit 
\begin{equation}\label{eq:limit-int}
  \lim_{\mu \to 0}  c_{\mu-\f12} \int_{-1}^1 f(t) (1-t^2)^{\mu-1} \d t = \frac{f(1) + f(-1)}{2}.  
\end{equation}

\subsection{Orthogonal structure of the solid paraboloid}
For $\b >  - \f{d+1}{2}$, $\g > -1$ and $\mu> -\f12$,  we define a weight function $W_{\b,\g,\mu}$ 
on $\VV^{d+1}$,
$$
    W_{\b,\g,\mu}(x,t) :=  t^\b (1-t)^\g (t - \|x\|^2)^{\mu-\f12}, \qquad  (x,t) \in \VV^{d+1}. 
$$
With respect to this weight function, we define an inner product 
$$
  \la f,g\ra_{\b,\g,\mu} = \bb_{\b,\g,\mu} \int_{\VV^{d+1}} f(x,t) g(x,t) W_{\b,\g,\mu}(x,t) \d x\d t,
$$
where $\bb_{\b,\g,\mu} = b_\mu c_{\b+\mu + \frac{d-1}{2}, \g}$ with $b_\mu$ is the normalization constant
of $\varpi_\mu$ on the unit ball and $c_{\a,\g}$ is defined in \eqref{eq:JacobiNorm}. The weight function
$W_{\b,\g,\mu}$ can be written as 
\begin{equation}\label{eq:Wbgmu}
W_{\b,\g,\mu}(x,t) =  t^{\b+\mu-\f12} (1-t)^\g (1 - \|x'\|^2)^{\mu-\f12},\quad\hbox{with}\quad  x' = \frac{x}{\sqrt{t}} \in \BB^d.
\end{equation}
Hence, the value of the constant $\bb_{\b,\g,\mu}$ can be  verified by using the identity 
$$
  \int_{\VV^{d+1} } f(x,t) \d x \d t = \int_0^1 \int_{\|x\|^2 \le t} f(x,t) \d x \d t = 
     \int_0^1 t^{\frac{d}{2}} \int_{\BB^d} f\big(\sqrt{t} y, t\big) \d y \d t.
$$

For $n =0,1,2,\ldots,$ let $\CV_n(\VV^{d+1}, W_{\b,\g,\mu})$ denote the space of orthogonal polynomials 
of degree $n$ in $(x,t)$ variables with respect to the inner 
product 
$ \la \cdot, \cdot\ra_{\b,\g,\mu}$ on the paraboloid. Then $\dim \CV_n(\VV^{d+1}, W_{\b,\g,\mu}) = 
\binom{n+d}{n}$. An orthogonal basis of this space can be given in terms of the Jacobi polynomials
and the classical orthogonal polynomials on the unit ball \cite{OX1}. We will also need the norms of 
these orthogonal polynomials. 

\begin{prop}
Let $\b > - \f{d+1}{2}$ and $\g > -1$. Let $\{P_{\kb}^m: |\kb| = m, \, \kb\in \NN_0^d\}$ denote an orthonormal 
basis with parity of $\CV_m^d(\BB^d,\varpi_\mu)$. For $0 \le m \le n$, define 
\begin{equation} \label{eq:OPbasisVV}
 \Qb_{m,\kb}^n(x,t) = P_{n-m}^{(m+\b+\mu+\f{d-1}{2},\g)}(1-2 t) t^{\f{m}{2}} 
    P_{\kb}^m \left(\frac{x}{\sqrt{t} }\right), \quad |\kb| = m, \, 0\le m \le n.
\end{equation}
Then $\{\sQ_{m,\kb}^n: |\kb | = m, \, 0 \le m  \le n, \, \kb \in \NN_0^d\}$ is an 
orthogonal basis of $\CV_n(\VV^{d+1}, W_{\b,\g,\mu})$. Moreover, the norm square of $\sQ_{m,\ell}^n$ is given by 
\begin{equation} \label{eq:OPbasisVVNorm}
   \hb_{m,n}^{\b,\g,\mu}  = \la \Qb_{m,\kb}^n, \Qb_{m,\kb}^n \ra_{\b,\g,\mu} 
       =  \frac{c_{\b+\mu+\f{d-1}{2},\g}}{c_{m+\b+\mu+\f{d-1}{2},\g} } h_{n-m}^{(m+\b+\mu+\f{d-1}{2},\g)}.
\end{equation}
\end{prop}

\begin{proof}
Using the parity of $P_{\kb}^m$, it is not difficult to see that $\Qb_{m,\kb}^n$ is a polynomial 
of degree $n$ in $(x,t)$ variables.
Let $\a = \b+\mu+\f{d-1}{2}$. Setting $y = x/\sqrt{t} \in \BB^d$, we obtain 
\begin{align*}
    \bb_{\b,\g,\mu}  \int_{\VV^{d+1} } & \Qb_{m,\kb}^n(x,t)  \Qb_{m',\kb'}^{n'}(x,t)W_{\b,\g,\mu}(x,t) \d x \d t  
     =  b_\mu  \int_{\BB^d}  P_\kb^m(y)P_{\kb'}^{m'}(y) \varpi_\mu(y) \d y \\
    &  \times c_{\a,\g} \int_0^1 P_{n-m}^{(m + \a,\g)}(1-2 t)P_{n'-m'}^{(m'+\a,\g)}(1-2 t) 
        t^{\frac{m+m'}{2} +\a} (1-t)^\g \d t.
\end{align*}
Since $P_\kb^m$ are orthonormal, it follows that  the second integral in the right-hand side is non-zero 
only when $m = m'$, from which the orthogonality of $\Qb_{m,\kb}^n$ and the formula for 
$\hb_{m,n}^{\b,\g,\mu}$ follow
from the corresponding properties of the Jacobi polynomials.   
\end{proof}

We know that orthogonal polynomials on the solid cone and hyperboloid are eigenfunctions of a second order 
linear differential operator with the eigenvalues depending only on the degree of the polynomials \cite{X19,X20}. 
In particular, this means that all polynomials of degree $n$ are eigenfunctions independent of the choice of bases.
In contrast, the orthogonal 
polynomials on the solid paraboloid, as those on the surface of the paraboloid, do not posses this property.
For polynomials $\Qb_{m,\kb}^n$, we can find a differential operator for which the eigenvalues depend 
on $n$ and $m$ but not on $\kb$, as seen in the following analogue of Proposition \ref{prp:pde}, where 
we assume $\b =0$. The latter assumption is consistent with $\b=-\f12$ in Proposition \ref{prp:pde} 
because $W_{\b,\g,\mu}$ contains the factor $t^{\b+\mu-\f12}$ when writing in the form \eqref{eq:Wbgmu} 
and, for $\mu =0$, $\varpi_0$ is the Chebyshev weight function on the unit ball, or the projection of the 
surface measure of $\SS^d$ onto $\BB^d$. 

\begin{prop}
Let $\b =0$, $\g > -1$ and $\mu > -\f12$. Then $\Qb_{m,\kb}^n$ in \eqref{eq:OPbasisVV} satisfies the differential 
equation 
\begin{align}\label{eq:diff-eqnVV}
 & \left[ t (1-t) \partial_t^2 + (1-t) \la x,\nabla_x\ra \partial_t + \frac{1}{4} (1-t)\Delta_x  \right. \\
      & \qquad \qquad \qquad  
      +\left.   \left(\mu+\tfrac{d+1}{2}\right)(1-t)\partial_t   - \frac{\g+1}{2} (2t \partial_t +  \la x,\nabla_x\ra) \right] u \notag \\
  & \qquad \qquad \qquad \qquad 
   = -  \left(n \big( n+\mu+ \g + \tfrac{d+1}{2}\big) - m \big(n + \mu+\tfrac{\g + d}{2}\big) \right)u.  \notag
\end{align}
\end{prop}

\begin{proof}
Let  $\a = \b+ \mu + \frac{d-1}{2}$. Set $g(t) = P_{n-m}^{(m+\a,\g)}(1-2t)$ and 
$H(x,t) = t^{\frac{m}{2}} P_{\kb}^m(\frac{x}{\sqrt{t}})$, so that $\Qb_{m,\kb}^n(x,t) = g(t)H(x,t)$. 
Since $H(x,s^2)$ is a homogeneous polynomial of degree $m$ in $(x,s)$ and, for $t = s^2$, 
$2 \sqrt{t} \frac{\partial}{\partial t}  = \frac{\partial}{\partial s}$, it follows by Euler's formula for 
homogenous polynomials that
\begin{equation}\label{eq:DtH}
  \left(2 t  \frac{\partial}{\partial t} + \la x, \nabla_x \ra \right)H = m H.
\end{equation}
Furthermore, since $\sqrt{t} \frac{\partial}{\partial t} H(x,t) = \frac{\partial}{\partial s} H(x,s^2)$ is a 
homogeneous polynomial of degree $m-1$ in $(x,s^2)$, applying \eqref{eq:DtH} on 
$\sqrt{t} \frac{\partial}{\partial t} H$ and simplifying gives 
\begin{equation}\label{eq:D2tH}
  2 t  \frac{\partial^2 H}{\partial t^2} + \la x, \nabla_x \ra \frac{\partial H}{\partial t}= (m-2)\frac{\partial H}{\partial t}.
\end{equation}
Let $u = \Qb_{m,\kb}^n$. Then $u = g(t) H(x,t)$. Taking derivative by the chain rule, 
a straightforward computation, using \eqref{eq:DtH} once, shows that
\begin{align*}
  t(1-t) \partial_{tt} u + (1-t) \la x,\nabla_x \ra \partial_t u
      =& \big( t(1-t) g''(t)  + m (1-t) g'(t) \big) H \\
          & + (1-t) g(t) \left( t  \frac{\partial^2 H}{\partial t^2} 
            + \la x, \nabla_x \ra \frac{\partial H}{\partial t} \right).
\end{align*}
The Jacobi polynomial satisfies the differential equation, so that $g$ satisfies  \eqref{eq:Jacobi-diff} with
$\a$ replaced by $\a + m$  and $n$ replaced by $n-m$, which leads to 
\begin{align}\label{eq:4.5}
    t(1-t) \partial_{tt} u +&\, (1-t) \la x,\nabla_x \ra \partial_t u + (1+\a -(\a+\g+2) t) \partial_t u \\
      = & - (n-m)(n+\a+\g+1) u - (\g+1) t g(t)  \frac{\partial H}{\partial t}  \notag \\
          & + (1-t) g(t) \left[ t  \frac{\partial^2 H}{\partial t^2} 
            + \la x, \nabla_x \ra \frac{\partial H}{\partial t} + (1+\a)  \frac{\partial H}{\partial t}  \right]. \notag
\end{align}
The polynomial $P_{\kb}^m$ satisfies a second order differential equation \eqref{eq:diffBall} with $n$ 
replaced by $m$, from which follows that $H$ satisfies
$$
   \left(t \Delta_x - \la x,\nabla_x\ra^2 - (2\mu + d-1) \la x,\nabla_x \ra\right) H = -m(m+2\mu+d-1)H. 
$$
Now, applying \eqref{eq:DtH} and \eqref{eq:D2tH}, the square bracket in the right-hand side of \eqref{eq:4.5}
satisfies
\begin{align*}
[ ...] \, &= \frac1{4t} \left( \la x,\nabla\ra + 2\a+m \right) \left(m - \la x ,\nabla_x \ra \right)H\\
          & =  \frac1{4t} \left(- \la x,\nabla_x\ra^2 - 2 \a \la x,\nabla_x \ra+ m(m+2 \a) \right)H
           = - \frac14 \Delta_x H, 
\end{align*}
where in the last step we have used $2\a = 2\mu + d-1$ for $\b =0$ and the differential equation satisfied
by $H$. Substituting this into \eqref{eq:4.5} and using \eqref{eq:DtH} one more time, the resulted identity 
is simplified to give \eqref{eq:diff-eqnVV}. 
\end{proof}

Next we consider the reproducing kernel of $\CV_n(\VV^{d+1}, W_{\b,\g,\mu})$. In terms of the basis
\eqref{eq:OPbasisVV}, the kernel is given by 
\begin{align*}
  \Pb_n\left(W_{\b,\g,\mu}; (x,t),(y,s)\right) = \sum_{m=0}^n \sum_{|\kb| =m} 
        \frac{ \Qb_{m,\kb}^n(x,t)  \Qb_{m,\kb}^n(y,s)}{\hb_{m,n}^{\b,\g,\mu}}.
\end{align*}   
With the help of the addition formula for the orthogonal polynomials on the unit ball, we can express this 
kernel in terms of the reproducing kernel $\Pb_n(U_{a,b}; \cdot,\cdot)$ on the parabola domain $\UU$. 

\begin{thm}
Let $d \ge 2$, $\mu \ge 0$ and $\g> -1$. Let $\a = \b+\mu+\f{d-1}{2}$. Then, for $(x,t) \in \VV^{d+1}$, 
$(y,s) \in \VV^{d+1}$ and $\b > 0$,
\begin{align} \label{eq:reprodSV1}
  \Pb_n & \left(W_{\b,\g,\mu}; (x,t),(y,s)\right)  = 
 c  \int_{[-1,1]^3} 
   \Pb_n \left(U_{\g, \a}; \big(\sqrt{t}, t\big), \big(\sqrt{s} \xi(x, t, y,s; z, u),s\big) \right) \\
        &\qquad \qquad\qquad \times (1-z_1)^{\mu+\frac{d-1}{2}} (1+z_1)^{\b-1}  (1-z_2^2)^{\b-\f12} 
           (1-u^2)^{\mu-1} \d z \d u,    \notag
\end{align}
where $c = c_{ \mu+ \frac{d-1}{2},\b-1} c_{\b}  c_\mu$ and 
\begin{align*}
  & \xi(x, t, y, s; z, u)  = \tfrac{1-z_1}{2 } \xi_0(x,t,y,s;u)+  \tfrac{1+z_1}2 z_2,\\
   & \quad \hbox{with} \quad   \xi_0(x, t, y,s; u) = \frac{1}{\sqrt{s t}}  \left( \la x,y\ra  + u \sqrt{t -\|x\|^2}\sqrt{s- \|y\|^2} \right);
\end{align*}
furthermore, for $\b = 0$, 
\begin{align} \label{eq:reprodSV2}
 \Pb_n& \left(W_{0,\g,\mu}; (x,t),(y,s)\right) \\
  & =   c_\mu \int_{-1}^1  
   \Pb_n\left(U_{\g,\mu+\frac{d-1}{2}}; \big(\sqrt{t}, t\big), \big(\sqrt{s} \xi_0(x, t, y,s; u),s\big) \right) (1- u^2)^{\mu-1} \d u.
    \notag
\end{align}
In both cases, the identity holds for $\mu = 0$ under the limit \eqref{eq:limit-int}.
\end{thm}

\begin{proof}
Since $\|x\| \le t$ and $\|y\| \le s$, we see that $|\xi_0(x,t,,y,s,u)|\le 1$ by the Cauchy inequality. Consequently,
$|\xi(x,t,y,s;z,u)|\le 1$, so that both variables in $\Pb_n( U_{\g,\b+\frac{d-1}{2}})$ are elements of $\UU$. 

By \eqref{eq:OPbasisVV} and the assumption that $P_\kb^m$ is orthonormal, it follows from the addition
formula \eqref{eq:additionB} on the unit ball that, with $\a = \b+\mu+\frac{d-1}{2}$, 
\begin{align*}
   \Pb_n(W_{\b,\g,\mu}; (x,t),(y,s))& =  c_\mu \int_{-1}^1 \sum_{m=0}^n 
            \frac{P_{n-m}^{(\a+m,\g)}(1-2 t)P_{n-m}^{(\a+m,\g)}(1-2 s)} {\hb_{m,n}^{\b,\g,\mu}}  t^{\f{m}{2}} s^{\f{m}{2}} \\
       &\times Z_m^{\mu+\f{d-1}{2}} \left(\frac{\la x,y\ra}{\sqrt{st}} + u \sqrt{1-\frac{\|x\|^2}{t}}\sqrt{1-\frac{\|y\|^2}{s}}\right) 
         (1- u^2)^{\mu-1} \d u.
\end{align*}   
If $\b = 0$, then $\a = \mu + \frac{d-1}{2}$, so that the sum under the integral sign can be identified with 
$\Pb_n(U_{\g, \a})$ by \eqref{eq:PnUab} with $x_2= t$, $y_2 = s$ and  $y_1 =  \xi_0(x, t, y,s; u)$. This proves
\eqref{eq:reprodSV2}. For $\b > 0$, we increase the value of the index of $Z_m^b$ from $\mu+\f{d-1}{2}$ 
to $\a = \mu + \b + \f{d-1}{2}$ by \eqref{eq:ZtoZ} with $\l = \mu+ \frac{d-1}{2}$ and $\s = \b$, so that the 
sum under the integral sign becomes 
\begin{align*}
   &  c_{ \mu+ \frac{d-1}{2},\b-1}  c_{\b}  \int_{-1}^1\int_{-1}^1 
    \sum_{m=0}^n \frac{P_{n-m}^{(\a+m,\g)}(1-2 t)P_{n-m}^{(\a+m,\g)}(1-2 s)} {\hb_{m,n}^{\b,\g,\mu}} 
         t^{\f{m}{2}} s^{\f{m}{2}} \\
      & \qquad \times Z_m^{\a} \big( \tfrac{1-z_1}{2 \sqrt{s}} \xi_0(x,t,y,s;u)+ \tfrac{1+z_1}2 z_2\big)
       (1-z_1)^{\mu+\frac{d-1}{2}} (1+z_1)^{\b-1}  (1-z_2^2)^{\b-\f12}  \d z \\
     &   =\, c_{ \mu+ \frac{d-1}{2},\b-1} c_{\b}  \int_{-1}^1\int_{-1}^1 \Pb_n \left(U_{\g, \a};
            \big(\sqrt{t}, t\big), \big(\xi(x, t, y,s; z, u),s\big) \right) \\
        &\qquad \qquad\qquad\qquad\qquad \times (1-z_1)^{\mu+\frac{d-1}{2}} (1+z_1)^{\b-1}  (1-z_2^2)^{\b-\f12}  \d z,        
\end{align*}
where the second step follows from \eqref{eq:OPbasisVVNorm} and \eqref{eq:PnUab}. Putting the
two displayed identities together, we have proved \eqref{eq:reprodSV1}. 
\end{proof}

If we allow $d =1$, then $\VV^2$ with $W_{0,\g,\mu}(x_1,x_2)$ should just be the domain $\UU$ with
$U_{\g, \mu}(x_1,x_2)$. We know that $\Pb_n(U_{\a,b})$ does not have a closed formula except in the
case that one of its variable is $\one = (1,1)$. For $d \ge 2$, we can deduce accordingly a closed formula 
for $\Pb_n\big(W_{\b,\g,\mu})$ on the hyperplane $t = 1$ of $(x,t) \in \VV^{d+1}$. We state this formula
for the kernel of the partial sum operator 
$$
\Kb_n\big(W_{\b,\g,\mu}, (x,t), (y,s)\big) = \sum_{m=0}^n \Pb_m\big(W_{\b,\g,\mu}; (x,t), (y,s)\big)
$$
by using the closed formula of $\Kb_n(U_{a,b}; \one, \cdot)$ in Theorem \ref{thm:Kn(1,x)}. We again state
the result only for the case $\b = 0$, for which the formula  is relatively simple. 

\begin{cor}
Let $d \ge 2$, $\g > -1$ and $\mu \ge 0$. Let $\tau = \mu+\f{d-1}{2}$. Then, for $x \in \BB^d$, 
\begin{align*}
  \Kb_n \big(W_{0,\g,\mu}; &  (x,1), (y,s) \big)   =  \frac{P_n^{(\g+\tau+1,\tau)}(1)}
     {h_n^{(\g+\tau+1,\tau)}}
     c_{\g+\tau+1,\tau} c_\mu \\
  & \times \int_{[-1,1]^2} P_n^{(\g+\tau+1,\tau)}\big(z'(x, y, s,u, v)\big)
     w_{\g+\tau1,\tau}(v) (1-u^2)^{\mu-1} \d u \d v, 
\end{align*}
where $z'(x, y, s,u, v) = z\left( \big(\sqrt{s} \xi_0(x,1,y,s; u), s \big), v\right)$ or 
$$
 z'(x, y, s, u, v) =  1 - (1-v^2) \left(1- \sqrt{s} \xi_0(x,1,y,s; u)\right) - \frac12 (1-v)^2 (1- s).
$$
\end{cor}
 
This is a corollary of \eqref{eq:reprodSV2} and \eqref{eq:Kn(1,x)}. A more involved closed form formula
for $\b > 0$ can be written down using \eqref{eq:reprodSV1} and \eqref{eq:Kn(1,x)}.

\subsection{Summability of Fourier orthogonal series}
Denote by $\proj_n(W_{\b,\g,\mu})$ the orthogonal projection operator 
$$
 \proj_n(W_{\b,\g,\mu}): L^2(\VV^{d+1}; W_{\b,\g,\mu})
   \mapsto \CV_n(\VV^{d+1}, W_{\b,\g,\mu}).
$$
In terms of the reproducing kernel $\Pb_n(W_{\b,\g,\mu})$ of $\CV_n(\VV^{d+1}, W_{\b,\g,\mu})$, we can write
$$
 \proj_n (W_{\b,\g,\mu}; f) = \sb_{\b,\g,\mu} \int_{\VV_0^{d+1}} f(y,s) \Pb_n\big(W_{\b,\g,\mu};
      (x,t), (y,s)\big) W_{\b,\g,\mu} (x,t)\d y \d s.
$$
For $f\in L^2(\VV^{d+1}; W_{\b,\g,\mu})$, the Fourier orthogonal series of $f$ on $\VV^{d+1}$ is defined by
$$
          f = \sum_{n=0}^\infty \proj_n ( W_{\b,\g,\mu}; f).
$$

Recall that $\UU_0$ denotes the curved portion of the boundary of the parabola domain $\UU$. 
In analogue to the Definition \ref{def:Tbg-V0}, we give the following definition: 

\begin{defn}
Let $d \ge 2$, $\mu \ge 0$, $\b \ge 0$ and $\g> -1$. Set $\a = \b+\mu+\f{d-1}{2}$. 
Let $g: \UU_0\times \UU \mapsto \RR$ such that, for each $t \in [0,1]$, the function 
$x \mapsto g((\sqrt{t},t); x)$ is in $L^1\big(\UU;U_{\g, \a}\big)$. 
For $(x,t)  \in \VV^{d+1}$ and $(y,s)  \in \VV^{d+1}$, define for $\b > 0$
\begin{align*}
   \Tb_{\b,\g,\mu} g\big( (x,t), (y,s)\big):= & 
       c \int_{[-1,1]^3} 
          g \left(\big(\sqrt{t}, t\big), \big(\sqrt{s} \xi(x, t, y,s; z, u),s\big) \right) \\
        &  \times (1-z_1)^{\mu+\frac{d-1}{2}} (1+z_1)^{\b-1}  (1-z_2^2)^{\b-\f12} 
           (1-u^2)^{\mu-1} \d z \d u,    \notag
\end{align*}
where $c = c_{ \mu+ \frac{d-1}{2},\b-1} c_{\b}  c_\mu$, and define for $\b =0$,
$$
    \Tb_{0,\g,\mu} g\big( (x,t), (y,s)\big) :=  c_\mu \int_{-1}^1  
  g\left(\big(\sqrt{t}, t\big), \big(\sqrt{s} \xi_0(x, t, y,s; u),s\big) \right) (1- u^2)^{\mu-1} \d u.
$$
In both cases the definition holds under the limit \eqref{eq:limit-int} when $\mu= 0$. 
\end{defn}

By \eqref{eq:reprodSV1} and  \eqref{eq:reprodSV2}, the definition $\Tb_{\b,\g,\mu}$ is motivated by 
\begin{equation}\label{eq:PnVV=PnU}
   \Pb_n (W_{\b,\g,\mu}; (x,t), (y,s)) =  \Tb_{\b,\g,\mu} \Pb_n (U_{\g, \b+\mu+\frac{d-1}{2}}) \big( (x,t), (y,s)\big).
\end{equation}
For each fixed $(x,t) \in \VV^{d+1}$, this is a bounded operator as seen below. 

\begin{prop}
Let $\b \ge 0$, $\mu \ge 0$ and $\g > -1$. Let $g: \UU_0\times \UU \mapsto \RR$ such that, for each $t \in [0,1]$, 
the function $g\big((\sqrt{t},t); \cdot\big)$ is in $L^1\big(\UU;U_{\g, \a}\big)$ with $\a =\b+\mu+\frac{d-1}{2}$.
Then, for $(x,t) \in \VV^{d+1}$, 
\begin{align}\label{eq:Tbg-bddVV}
   \int_{\VV^{d+1}}  \left|\Tb_{\b,\g,\mu} g\big((x,t), (y,s) \big) \right|& W_{\b,\g,\mu}(y,s) \d y \d s \\
    \le c  \int_{\UU} & \left|g \left(\big(\sqrt{t},t\big), z\right)\right| U_{\g,\a}(z) \d z. \notag
\end{align}
\end{prop}

\begin{proof}
We follow the approach for the proof of Proposition \ref{prop:Tbg-bdd}. Instead of the integral relation 
\eqref{eq:IntSph}, we use the following identity for $h: [-1,1]\mapsto \RR$ and $v \in \BB^d$, 
\begin{align}\label{eq:IntBkernel}
 b_\mu \int_{\BB^d} \int_{-1}^1  
 h(\la u,v\ra + \sqrt{1-\|u\|^2}& \sqrt{1-\|v\|^2}\, r) (1-r^2)^{\mu-1} \d r(1-\|u\|^2)^{\mu-\f12} \d u \\
   &  = c_{\mu+\f{d-1}{2}} \int_1^1 h(t) (1-t^2)^{\mu+\f{d-2}{2}} \d t. \notag
\end{align}
This identity is established in the proof of \cite[Theorem 5.3]{X99} for a specific function $h$ but the proof
clearly holds for all generic $h$. Let $G(z) = g( (t,\sqrt{t}), z)$ for $z \in \UU$. Then, in the case of $\b = 0$, 
we obtain 
\begin{align*}
  & \bb_{0,\g,\mu} \int_{\VV^{d+1}} \left|\Tb_{0,\g,\mu} g\big((x,t), (y,s) \big) \right| W_{0,\g,\mu}(y,s) \d y \d s \\
  & \quad = \bb_{0,\g,\mu} \int_0^1 \int_{\BB^d} 
       \left|\Tb_{0,\g,\mu} g\big((x,t), (\sqrt{s}y',s) \big) \right| (1-\|y'\|^2)^{\mu-\f12}  \d y'  s^{\mu+\f{d-1}{2}} (1-s)^\g \d s \\
 & \quad \le   
  \bb_{0,\g,\mu} \int_0^1 \int_{\BB^d}  c_\mu \int_{-1}^1  
   \left|G\left(\sqrt{s} \xi_0(x, t, \sqrt{s}y', s; r),s \right) \right|\\
       &\qquad\qquad\qquad\qquad\qquad \times  (1- r^2)^{\mu-1} \d r (1-\|y'\|^2)^{\mu-\f12}  \d y'  s^{\mu+\f{d-1}{2}}
           (1-s)^\g \d s.
\end{align*}
Since $ \xi_0(x, t, \sqrt{s}y', s; z) = \la x',y'\ra + z \sqrt{1-\|x'\|^2} \sqrt{1-\|y'\|^2}$ with $x' = x/\sqrt{t} \in \BB^d$ and
$y'\in \BB^d$, we can apply \eqref{eq:IntBkernel} to bound the above inequality by
\begin{align*}
   c \int_0^1 \int_{-1}^1  
   \left|G\left(\sqrt{s}u, s \right) \right| (1-u^2)^{\mu+\f{d-2}{2}}&  \d u  s^{\mu+\f{d-1}{2}} (1-s)^\g \d s \\
     & = \db_{\g,\mu+\f{d-1}{2}} \int_{\UU} |G(z)| U_{\g,\mu+\f{d-1}{2}}(z) \d z,
\end{align*}
which follows from changing variables $z_1 = \sqrt{s} u$ and $z_2 = s$ and the the last constant is determined 
by setting $G(z) =1$. Consequently, this establishes 
\eqref{eq:Tbg-bddVV} for $\b =0$.

Let now $\b > 0$. Following the proof in the case of $\b = 0$ by using \eqref{eq:IntBkernel}, we obtain
\begin{align*}
  & \bb_{\b,\g,\mu} \int_{\VV^{d+1}} \left|\Tb_{0,\g,\mu} g\big((x,t), (y,s) \big) \right| W_{\b,\g,\mu}(y,s) \d y \d s \\
 & \, \le   
  \bb_{\b,\g,\mu} \int_0^1  \int_{-1}^1  \int_{-1}^1 \int_{\BB^d}  c_\mu \int_{-1}^1  
   \left|G\left(\sqrt{s} \xi(x, t, \sqrt{s}y', s; z, r), s \right) \right|  (1- r^2)^{\mu-1} \d r    \\
       &\,\quad  \times (1-\|y'\|^2)^{\mu-\f12}  \d y'
     (1-z_1)^{\mu+\frac{d-1}{2}}(1+z_1)^{\b-1}(1-z_2)^{\b-\f12} \d z s^{\b+\mu+\f{d-1}{2}} 
           (1-s)^\g \d s \\
   & \, \le   c \int_0^1  \int_{-1}^1  \int_{-1}^1 \int_{-1}^1
         \left|G\left(\sqrt{s} \big(\tfrac{1-z_1}{2} u + \tfrac{1+z_1}{2} z_2 \big), s \right) \right|  (1-u^2)^{\mu+\f{d-2}{2}} \d u \\
       & \qquad\qquad\qquad \times (1-z_1)^{\mu+\frac{d-1}{2}}(1+z_1)^{\b-1}(1-z_2)^{\b-\f12} 
           s^{\b+\mu+\f{d-1}{2}}  (1-s)^\g \d z  \d s.        
\end{align*}
Apart from the difference in their parameters, the last integral is the same as the quadruple integral in the 
proof of Proposition \ref{prop:Tbg-bdd}. Indeed, if we replace $\b $ by $\b+\f12$ and $\mu + \frac{d-1}{2}$
by $\f{d-2}{2}$ in the above integral, then the two quadruple integrals are the same. Hence, we can estimate
the above integral as in the proof of Proposition \ref{prop:Tbg-bdd} to complete the proof of \eqref{eq:Tbg-bddVV} 
for $\b > 0$. This completes the proof.
\end{proof}

As in the case on the surface $\VV_0^{d+1}$, we define a pseudo convolution on $\VV^{d+1}$.
 
\begin{defn}
Let $\b \ge 0$, $\mu \ge 0$ and $\g > -1$. Let $g: \UU_0\times \UU \mapsto \RR$ such that, for each $t \in [0,1]$, 
the function $x \mapsto g((\sqrt{t},t); x)$ is in $L^1\big(\UU; U_{\g, ,\b+\mu+\f{d-1}2}\big)$. 
For $f \in L^1(\VV^{d+1}; W_{\b,\g,\mu})$ and $(y,s) \in \VV^{d+1}$, define
$$
  (f\ast_{\VV} g)(y,s) = \bb_{\b,\g,\mu}\int_{\VV^{d+1}} f(x,t) \Tb_{\b,\g,\mu} g\big( (x,t), (y,s) \big) 
  W_{\b,\g,\mu}(x,t) \d x \d t.
$$
\end{defn}

By \eqref{eq:PnVV=PnU}, it follows that the project operator onto $\CV_n(\VV^{d+1}, W_{\b,\g,\mu})$
satisfies
$$
  \proj_n(W_{\b,\g,\mu}; f) = f \ast_{\VV} \Pb_n\left(U_{\g,\b+\mu+\f{d-1}2}\right).
$$

\begin{prop} \label{prop:f*gV}
Let $\b \ge 0$, $\mu \ge 0$ and $\g > -1$. Let $\a = \b+\mu + \f{d-1}2$. For $f \in L^p (\VV^{d+1};W_{\b,\g,\mu})$,
$1 \le p < \infty$, and $f \in C(\VV^{d+1})$ for $p = \infty$, 
$$
    \| f\ast_{\VV} g\|_{L^p (\VV^{d+1};W_{\b,\g,\mu})} \le c \|f\|_{L^p (\VV^{d+1};W_{\b,\g,\mu})}
      \max_{t\in [0,1]} \int_{\UU} \left|g\left( \big(\sqrt{t}, t), z \right)\right| U_{\g,\a}(z) \d z.
$$
\end{prop}

\begin{proof}
Using \eqref{eq:Tbg-bddVV}, the proof follows exactly as that of Proposition \ref{prop:f*gV0}.
\end{proof}

We now use this boundedness of the pseudo convolution to study the Ces\`aro means of the Fourier 
orthogonal series on the solid paraboloid. 

For $\delta > -1$, let $\Kb_n^\delta((x,t),(y,s))$ be the kernel of 
the Ces\`aro $(C,\delta)$ means $\Sb_n^\delta \big(W_{\b,\g,\mu}; f)$. In analogue to \eqref{eq:KdV=KdU}, 
we derive from \eqref{eq:PnVV=PnU} that 
\begin{equation}\label{eq:KdVV=KdU}
  \Kb_n^\delta (W_{\b,\g,\mu}; (x,t), (y,s)) =  \Tb_{\b,\g,\mu} \Kb_n^\delta \big(U_{\g, \b+\mu+\frac{d-1}{2}}\big) 
    \big( (x,t), (y,s)\big).
\end{equation}
Furthermore, in terms of the pseudo convolution, we can write 
$$
\Sb_n^\delta \big(W_{\b,\g,\mu}; f, (x,t) \big) = f \ast_{\VV} \Kb_n^\delta \big(U_{\g, \b+\mu+\frac{d-1}{2}}\big)(x,t).
$$

\begin{thm}
Let $d\ge 2$, $\b \ge 0$, $\mu \ge 0$ and $\g> -1$. If $f \in C(\VV^{d+1})$, then the Ces\`aro means
$\Sb_n^\delta\big(W_{\b,\g,\mu}; f, (x,1)\big)$ converge to $f(x, 1)$ uniformly for $x \in \BB^d$ 
provided $\delta > \b+\g+ \mu+ \f{d+2}{2}$.
\end{thm}

\begin{proof}
The convergence of $\Sb_n^\delta\big(W_{\b,\g,\mu}; f, (x,1)\big)$ holds if and only if 
$$
 \sup_{x \in \BB^d} \int_{\VV^{d+1}} \left| \Kb_n^\delta (W_{\b,\g,\mu}; (x,1), (y,s)) \right| W_{\b,\g,\mu}(y,s) \d y \d s 
$$ 
is bounded uniformly in $n$. By \eqref{eq:KdVV=KdU}, the fact that $t=1$ and the inequality 
\eqref{eq:Tbg-bddVV} shows that this follows from the boundedness of the $L^1$ norm of 
$\Kb_n^\delta \big(U_{\g,\b+\mu+\f{d-1}{2}}; \one, z)$, which holds for $\delta > \g+\b+\mu+\f{d-1}{2} +\f32$
by Theorem \ref{thm:CesaroU} and by Remark \ref{rem:CesaroU}.
\end{proof}

\begin{thm}
Let $d\ge 2$, $\b \ge 0$ and $\mu \ge 0$, $\g \ge \b +\mu + \f{d-1}{2}$. Let $f \in L^p(\VV^{d+1}, W_{\b,\g,\mu})$ 
for $1 \le p < \infty$ and $f \in C(\VV^{d+1})$ for $p = \infty$. Then the Ces\`aro means 
$\Sb_n^\delta\big(W_{\b,\g,\mu}; f)$ satisfy 
\begin{enumerate}[ 1.]
\item if $\delta \ge 2 \b +2 \mu+ \g+ d+3$, then $\Sb_n^\delta (W_{\b,\g,\mu}; f)$ is nonnegative if $f$ is nonnegative; 
\item if $\delta > \b+\mu+ \g+ \f{d+2}2$, then $\Sb_n^\delta (W_{\b,\g,\mu}; f)$ converge to $f$ in 
$L^p(\VV^{d+1}; W_{\b,\g,\mu})$, $1 \le p < \infty$, and in $C(\VV^{d+1})$.   
\end{enumerate}
\end{thm}
 
\begin{proof}
Using the identity \eqref{eq:KdVV=KdU}, the proof reduces to properties possessed by  
the kernel $\Kb_n^\delta (U_{\g, \b+\mu+\frac{d-1}{2}})$ on the parabola domain $\UU$. The detail
follows exactly as in the proof of Theorem \ref{thm:(C,d)V0} and we leave it to interested readers. 
\end{proof}

\end{document}